\documentclass[12pt]{amsart}

\usepackage{amsfonts,amsmath,amssymb,graphicx,mathtools,amsthm,latexsym,amsxtra}
\usepackage{amscd,tikz,tikz-cd}
\usepackage{fancyhdr,xcolor,upref}
\usepackage{float,array,bm,booktabs,adjustbox,diagbox,longtable,multirow}
\usepackage{xy,mathrsfs,verbatim,colonequals,enumerate,dsfont,lipsum}
\usepackage[left=1in, right=1in, top=1in, bottom=1in]{geometry}

\usepackage[tableposition=above]{caption}
\definecolor{navy}{HTML}{2F729C}
\definecolor{red1}{HTML}{FF0000}
\usepackage{hyperref}
\hypersetup{nesting=true,debug=true,naturalnames=true,colorlinks=true,urlcolor=magenta,citecolor=red,linkcolor=blue}

\let\<\langle
\let\>\rangle

\let\uml\"

\newcolumntype{C}[1]{>{\centering\arraybackslash}p{#1}}

\usepackage[mathcal]{euscript}

\newcommand{\Fl}{\mathbb{F}_{\ell}}

\newcommand{\Zl}{\mathbb{Z}_{\ell}}
\newcommand{\im}{\operatorname{im}}

\newcommand{\Z}{\mathbb{Z}}
\newcommand{\F}{\mathbb{F}}

\newcommand{\Q}{\mathbb{Q}}

\newcommand{\GL}{\operatorname{GL}}

\newcommand{\ddef}{\colonequals}

\newcommand{\tors}{\mathsf{tors}}

\newcommand{\abs}[1]{\lvert {#1} \rvert}

\numberwithin{equation}{section}

\theoremstyle{plain}
\newtheorem{thm}[equation]{Theorem}

\newtheorem{lem}[equation]{Lemma}

\newtheorem{cor}[equation]{Corollary}
\newtheorem{prop}[equation]{Proposition}

\theoremstyle{remark}
\newtheorem{rmk}[equation]{Remark}

\newtheorem{exm}[equation]{Example}

 \input xy
 \xyoption{all}

\begin{document}

\title[Tamagawa Numbers of Elliptic Curves with an $\ell$-isogeny]{Tamagawa Numbers of Elliptic Curves \\ with an $\ell$-isogeny}

\author{Alexander J. Barrios}
\address{Department of Mathematics, University of St. Thomas, St. Paul, MN 55105, USA}
\email{abarrios@stthomas.edu}
\urladdr{\url{https://sites.google.com/site/barriosalex/home}}

\author{John Cullinan}
\address{Department of Mathematics, Bard College, Annandale-On-Hudson, NY 12504, USA}
\email{cullinan@bard.edu}
\urladdr{\url{http://faculty.bard.edu/cullinan/}}


\begin{abstract}
Let $\ell$ be an odd prime, and suppose $E$ is an elliptic curve defined over the rational numbers $\Q$.  If $E$ has an $\ell$-torsion point, then there has been significant work done on characterizing the $\ell$-divisibility of the global Tamagawa number of $E$.  In this paper, we consider elliptic curves that are $\ell$-isogenous to elliptic curves with an $\ell$-torsion point and study the $\ell$-divisibility of their global Tamagawa numbers.
\end{abstract}

\maketitle

\section{Introduction}

\subsection{Setup and Motivation} Let $E$ be an elliptic curve defined over the rational numbers $\Q$.  Then $E$ admits finitely many rational isogenies $E \to E'$, where $E'$ is another elliptic curve over $\Q$.  The set of isomorphism classes of isogenous elliptic curves to $E$ form the vertices of a graph (the \emph{isogeny graph} of $E$), with edges corresponding to prime-degree isogenies.  Isogeny graphs of elliptic curves over $\Q$ were recently explicitly classified in \cite{barrios_ig}. \emph{Isogeny-torsion} graphs are isogeny graphs whose vertices are decorated with the torsion subgroup of the elliptic curve. In  \cite{MR4203041}, all possible isogeny-torsion graphs over $\Q$ were classified. It is an interesting number theoretic problem to study how the arithmetic properties of elliptic curves vary over members of the same graph.  

Here is the setup we are considering. Fix a positive integer $m$ and suppose $|E(\Q)_\tors| = m$. For each prime $\ell$ at which $E$ has good reduction, every $\Q$-isogenous elliptic curve $E'$ to $E$ has the property that $|E'(\F_\ell)|$ is divisible by $m$. This is due to: (1) the well-known injectivity of rational torsion into $E(\F_\ell)$ under reduction modulo $\ell$, (2) the fact that the number of $\F_\ell$-points is an isogeny invariant, and (3) that a $\Q$-isogeny descends to an $\F_\ell$-isogeny for all but finitely many $\ell$.  However, it is not necessarily the case that all of these $E'$ must have $|E'(\Q)_\tors| = m$.  Introducing a little more terminology as in \cite{ckv}, we say that $E/\Q$ \emph{locally has a subgroup of order $m$} if $|E(\F_\ell)| \equiv 0 \pmod{m}$ for all good primes $\ell$. 

\begin{exm}
The isogeny class \href{https://www.lmfdb.org/EllipticCurve/Q/880/h/}{880.h} \cite{lmfdb} consists of four elliptic curves:
\begin{center}
\begin{tabular}{|ll|}
\hline
Curve & Torsion Subgroup \\
\href{https://www.lmfdb.org/EllipticCurve/Q/880/h/1}{880.h1} & $\Z/2\Z$ \\
\href{https://www.lmfdb.org/EllipticCurve/Q/880/5/2}{880.h2} & $\Z/4\Z$ \\
\href{https://www.lmfdb.org/EllipticCurve/Q/880/h/3}{880.h3} & $\Z/2\Z \times \Z/2\Z$ \\
\href{https://www.lmfdb.org/EllipticCurve/Q/880/h/4}{880.h4} & $\Z/2\Z$\\
\hline
\end{tabular}
\end{center}
The curves 880.h1 and 880.h4 locally have subgroups of order 4, but not globally.
\end{exm}

We immediately specialize to the case $m=\ell$ is a fixed, odd prime number for the remainder of the paper.  By Mazur's torsion theorem \cite{mazur}, we are necessarily restricted to $\ell \in \lbrace 3,5,7 \rbrace$.  If $E(\Q)$ has a point $P$ of order $\ell$, then set $\widetilde{E} \ddef E/\langle P \rangle$.  In terms of the mod~$\ell$ representation,  the difference between the two curves is that if we choose a basis for $E[\ell]$ such that
\begin{align} 
\im  \overline{\rho}_{E,\ell} = \left\{\begin{pmatrix} 1  & * \\ 0 & * \end{pmatrix}\right \}, \label{1*0*}
\end{align}
then up to an ordering of a corresponding basis for $\widetilde{E} [\ell]$, we have
\begin{align} 
\im \overline{\rho}_{\widetilde{E} ,\ell} = \left\{\begin{pmatrix} *  & * \\ 0 & 1 \end{pmatrix}\right\}.\label{**01}
\end{align}
If, furthermore, $\im \overline{\rho}_{E,\ell}$ is a split Cartan subgroup of $\GL_2(\Fl)$, then (at least one of) the $\ell$-isogenous curves to $E$ will have an $\ell$-torsion point (by \cite{zywina}, this restricts $\ell \in \lbrace 3,5 \rbrace$). 

Let $\mathcal{L}_{\ell}$ be the set of isomorphism classes of elliptic curves over $\Q$ that locally have a subgroup of order $\ell$ and $\mathcal{G}_{\ell} \subseteq \mathcal{L}_{\ell}$ the subset of curves that have a global point of order $\ell$.  (When $\ell = 2$, $\mathcal{G}_{2}  = \mathcal{L}_{2}$, but for odd $\ell$, $\mathcal{G}_{\ell}$ is a proper subset of $\mathcal{L}_{\ell}$,  another reason why we take $\ell$ odd.)  This brings us to the main idea of the current paper.  We are interested in how the $\ell$-divisibility of the  global Tamagawa number $c_E$ of elliptic curves differs between the sets $\mathcal{G}_{\ell}$ and $\mathcal{L}_{\ell} \smallsetminus \mathcal{G}_{\ell}$.

There has been a good amount of work done already on studying the $\ell$-divisibility of Tamagawa numbers of elliptic curves with a non-trivial $\ell$-torsion point.  For example, in \cite{lorenzini} and \cite{br}, the authors studied, among other things, the conditions under which $c_E$ is divisible by $\ell$ for $E \in \mathcal{G}_{\ell}$.  Specifically, in \cite{lorenzini}, it was shown that if $\ell = 7$, then $c_E$ is divisible by 7 for all $E/\Q$ with a rational 7-torsion point, and if $\ell =5$, then $c_E$ is divisible by 5 unless $E$ is isomorphic to the modular curve $X_1(11)$.  When $\ell \in \lbrace 2,3 \rbrace$, there are infinitely many $E/\Q$ with a rational $\ell$-torsion point with $c_E$ coprime to $\ell$, and in \cite{br} the authors give a complete classification of such curves for which $c_E=1$.  Recently, in \cite{melistas}, the Melistas showed that for every number field $K/\Q$, there exists a constant $n_{K}$ such that for every $\ell \geq 7$ and every $E \in \mathcal{G}_{\ell}$, it is the case that $c_{E/K} \equiv 0 \pmod{\ell}$ with at most $n_{K}$ exceptions. We also note that the ratio of the {local} Tamagawa numbers of $\ell$-isogenous elliptic curves is a power of $\ell$ \cite[Lemma 6.2]{Dokchitsersquared}. When $\ell\ge 3$, the ratio is $\ell^{\pm1}$. However, for $\ell=2$, the ratio could be $4^{\pm1}$ as illustrated by the isogeny class \href{https://www.lmfdb.org/EllipticCurve/Q/14400/cr/}{14400.cr}.

In this paper, we study the $\ell$-divisibility of global Tamagawa numbers for elliptic curves in $\mathcal{L}_{\ell} \smallsetminus \mathcal{G}_{\ell}$.  Loosely speaking, we find that the presence of a local subgroup of order $\ell$ has a more nuanced effect on the $\ell$-divisibility of the global Tamagawa number than does having a global point of order $\ell$. We now describe our results in detail. 

\subsection{Main Results}  Our main results  characterize the elliptic curves in $\mathcal{L}_{\ell} \smallsetminus \mathcal{G}_{\ell}$ whose global Tamagawa number is divisible by $\ell$, for $\ell \in \lbrace 3,5,7\rbrace$. Most of our treatment of parametric models and notational conventions is taken from \cite{Barrios}, and we briefly review it here.

Let $C_\ell$ denote the cyclic group of order $\ell$.  Then the elliptic curves over $\Q$  with torsion subgroup isomorphic to $C_\ell$ can be integrally parameterized, as determined in \cite{Barrios} (we review the explict models below).  For ease of notation, once $\ell$ is fixed we set $T = C_\ell$.  If $\ell\in \lbrace 5,7 \rbrace$, then all elliptic curves with torsion subgroup $C_\ell$ are given by the specialization of a two-parameter integral family $E_T(a,b)$.  When $\ell = 3$, we need to distinguish between $E_{C_3}(a,b)$ which do not have $j$-invariant 0, and $E_{C_3^0}(a)$, which have $j$-invariant 0.  To further ease notation, unless we have explicit need of the parameters, we write $E_T$ for $E_T(a,b)$.  

We set $\widetilde{E}_T \ddef E_T/T$; thus $\widetilde{E}_T$ is rationally $\ell$-isogenous to $E_T$.  It is possible that $\widetilde{E}_T$ has a rational $\ell$-torsion point when $\ell \in \lbrace 3,5\rbrace$, but it will never be the case that $\widetilde{E}_T$ has a 7-torsion point when $\ell = 7$.  We also denote the discriminants of $E_T$ and $\widetilde{E}_T$ by $\Delta_T$ and $\widetilde{\Delta}_T$, respectively. Similarly, we let $c_T$ and $\widetilde{c}_T$ denote the global Tamagawa numbers of $E_T$ and $\widetilde{E}_T$, respectively.  If it is clear from context, we may omit the subscript ``$T$'' from the global Tamagawa numbers.  

Our first result is a complete characterization of the local Tamagawa numbers of $E_T$ and $\widetilde{E}_T$, for all primes $p$, in terms of the integral parameters $a$ and $b$ (reviewed in Section~\ref{background_section}); this is Theorem~\ref{thmontama} in the main body of the paper, which we reproduce here without the referencing tables in order to save space.

\begin{thm}
There are necessary and sufficient conditions on the parameters of  $E_{T}$ and $\widetilde{E}_{T}$ to determine the local Tamagawa numbers $c_{p}$ and $\widetilde{c}_{p}$ of $E_{T}$ and $\widetilde{E}_{T}$, respectively, at each prime $p$ of bad reduction. These conditions are listed in Table~\ref{ta:Tamagawanumbers}.
\end{thm}

The proof of Theorem \ref{thmontama} involves an explicit case-by-case analysis of Tate's algorithm and the type and signature of the local singular point.  In the case when $T=C_3$, we need to deduce integral \emph{minimal} models for $E_T$ and $\widetilde{E}_T$, which involves working with additional integral parameters $c$, $d$, and $e$.

Next, we use the information gained in Theorem \ref{thmontama} to try to quantify how often $\widetilde{E}_T$ has global Tamagawa number divisible by $\ell$.  In terms of Galois representations, the difference between our work and \cite{lorenzini} is that ours is an analysis of elliptic curves with mod $\ell$ representation (\ref{**01}), while Lorenzini considered the case (\ref{1*0*}).  We now give a summary of these results, beginning with the case $\ell=3$.  In the main body of the paper, the following theorem is split into  Propositions \ref{PropTama3j0} and \ref{PropTama3}.

\begin{thm}\label{thmcis1lis3intro}
There are infinitely many elliptic curves $E/\Q$ with a 3-torsion point $P$ such that the global Tamagawa number $\widetilde{c}$ of $\widetilde{E}=E/\left\langle P\right\rangle $ is divisible by $3$.  There are also infinitely many elliptic curves $E/\Q$ with a 3-torsion point $P$ such that the global Tamagawa number $\widetilde{c}$ of $\widetilde{E}=E/\left\langle P\right\rangle $ is 1.
\end{thm}

In general, if an elliptic curve $E$ has a $\ell$-torsion point $P$, then the image of its mod $\ell$ representation lies in a Borel subgroup of $\GL_2(\F_\ell)$, as in (\ref{1*0*}).  If $\widetilde{E} = E/\langle P \rangle$ also has a 3-torsion point, then it is necessarily the case that the image of the mod $\ell$ image of $\widetilde{E}$ lies in a split Cartan subgroup.  In this special case, we have the following result, which appears later as Proposition \ref{split3}.

\begin{prop}\label{3cartanintro}
Let $E/\mathbb{Q}$ be an elliptic curve with a $3$-torsion point $P$ such that $\widetilde{E}
=E/\left\langle P\right\rangle $ also has a $3$-torsion point. Then, with the
exception of when $E$ is $\mathbb{Q}$-isomorphic to the elliptic curve \href{https://www.lmfdb.org/EllipticCurve/Q/27/a/2}{27.a2}, it is the case that the
global Tamagawa number of $\widetilde{E}$ is divisible by~$3$.
\end{prop}

Lorenzini \cite{lorenzini} showed that there are infinitely many elliptic curves with a $3$-torsion point with $c=1$. Such curves were explicitly classified in \cite{br}. These results, together with Theorem~\ref{thmcis1lis3intro}, lead us
 us to ask the question of whether there are isogeny classes of elliptic curves containing an elliptic curve $E$ with a $3$-torsion point $P$ such that the global Tamagawa number of both $E$ and $\widetilde{E}=E/\left\langle P\right\rangle $ is $1$. In this case, we have that there is exactly one isogeny class satisfying this criterion. Namely, the isogeny class with LMFDB label \href{https://www.lmfdb.org/EllipticCurve/Q/27/a}{27.a}. This isogeny class contains the elliptic curves \href{https://www.lmfdb.org/EllipticCurve/Q/27/a/2}{27.a2} and \href{https://www.lmfdb.org/EllipticCurve/Q/27/a/4}{27.a4}, which satisfy the above assumptions for $E$ and $\widetilde{E}$, respectively. Both elliptic curves have global Tamagawa number $1$. However, if we instead consider the product of the global Tamagawa numbers of isogenous elliptic curves, then we have the following result, which is Theorem \ref{LoreExte} in the article.

\begin{thm}\label{LoreExteintro}
Let $E$ be an elliptic curve with a $3$-torsion point $P$. Then, the product of the global Tamagawa numbers of the elliptic curves in the isogeny class of $E$ is divisible by~$3$.
\end{thm}

When $\ell \in \lbrace 5,7 \rbrace$, the local behavior is simpler to determine because away from $\ell$ the reduction is not additive \cite{MR0457444}, but we are unable to prove an infinitude result unconditionally.  We will give a briefer, itemized summary of these results since they are similar in spirit to the case $\ell=3$. 

\begin{itemize}
\item Roughly, if the parameters $a$ or $b$ has a factor that is an $\ell$-th power, then the curve $\widetilde{E}_T$ will have $\widetilde{c}$ divisible by $\ell$. When $\ell=5$, if $ab$ is not a $5$-th power, then $\widetilde{E}_T$ will only have a local $\ell$-subgroup, not a global.  We prove both of these results in Section~\ref{pthpowersection}.  
\item It is more difficult to prove that there are infinitely many $\widetilde{E}_T$ with $\widetilde{c}$ coprime to $\ell$.  While the statement is most likely true for both $\ell=5$ and $\ell=7$, we are only able to prove these results conditional on certain integral Diophantine equations having infinitely many points.  We make all of this precise in Section \ref{conditional_proofs}. 

\item Finally, in the case $\ell=5$, if $\widetilde{E}_T$ has a point of order $5$, then the image of the mod~$5$ representation of $\widetilde{E}_T$ is necessarily a split Cartan subgroup and we know by \cite{lorenzini} that except for the situation where one of $E$ or $\widetilde{E}$ is the curve \href{https://www.lmfdb.org/EllipticCurve/Q/11/a/3}{11.a3}, both will have global Tamagawa number divisible by 5.  In that case, we explore the 5-divisibility of the global Tamagawa numbers of $E$ and $\widetilde{E}$ in Corollary \ref{5cartan}.
\end{itemize}

\subsection{Future Directions} \label{future}

Given that we have explicit conditions for when the global Tamagawa number of $\widetilde{E}$ is divisible by $\ell$, one can try to go further and count the elliptic curves $E$ up to height $X$ such that the associated $\widetilde{c}$ is divisible by $\ell$.  This was the approach taken in \cite{hs}, and refined in \cite{ckv}, for counting elliptic curves with a given torsion structure over $\Q$.  But we can generalize even further. 

The groups (\ref{1*0*}) and (\ref{**01}) are just two possible mod $\ell$ representations attached to elliptic curves over $\Q$.   More generally, given an elliptic curve $E$,  we are interested in how the $\ell$-divisibility of its global Tamagawa number depends on the image $G_\ell$ of the $\ell$-adic representation.  Specifically, let $\mathcal{E}$ denote the set of isomorphism classes of elliptic curves over $\Q$ ordered by height; write $\mathcal{E}_{X}$ for the subset of $\mathcal{E}$ consisting of elliptic curves with height at most $X$.   

Fix a prime number $\ell$, a positive integer $n$, and a subgroup $G_\ell \subseteq \GL_2(\Zl)$.  Denote by 
$N(G_\ell,n,X)$ the number of $E \in \mathcal{E}_X$ such that $\im \rho_{E,\ell}$ is conjugate to  $G_\ell$ and $c_E \equiv 0 \pmod{\ell^n}$.  It would be interesting in a future work to determine asymptotics for $N(G_\ell,n,X)$ for various subgroups of $\GL_2(\Zl)$; our work in this paper, together with the work of \cite{br} and \cite{lorenzini} can be viewed as a start to this classification.   In this particular direction, we would need asymptotic estimates on the number of curves in $\mathcal{L}_{\ell} \smallsetminus \mathcal{G}_{\ell}$ with $c_E \equiv 0 \pmod{\ell}$, which we are currently unable to prove.

For a sample calculation, consider the universal family of elliptic curves with a $5$-torsion point, given explicitly by
\begin{align} \label{5tors}
E_t: \qquad y^2 + (1-t)xy -ty = x^3 - tx^2
\end{align}
for nonzero $t \in \Q$.  Observe that by writing $t = \frac{b}{a}$ in lowest terms and rescaling, $E_t$ matches the formula for $E_{C_5}(a,b)$ derived below.

For all $t \in \Q^{\times}$, $E_t$ has a rational point of order 5, and any elliptic curve over $\Q$ with a rational point of order 5 is a specialization of $E_t$.   The elliptic curve $E_1/\Q$ is a model for the modular curve $X_1(11)$ and its global Tamagawa number $c_{E_1} = 1$.  By \cite[Proposition~2.7]{lorenzini}, $E_1$ is the unique specialization of $E_t$ with $c_{E_t}$ coprime to 5. Now consider the elliptic curve $E'_t \ddef E_t/\langle (0,0) \rangle$ that is 5-isogenous to $E_t$ and, unless $t \in (\Q^{\times})^5$, $E'_t$ does not have a 5-torsion point.  

We perform a quick naive experiment: for integral $t \in [1,X]$, we count the number $N(X)$ of specializations $E'_t$ whose global Tamagawa number is divisible by 5.  
\begin{center}
\begin{tabular}{l|r|r|r|r|r|r}
$X$ & $10$ & $100$ & $1000$& $10000$ & $100000$ & $1000000$ \\
\hline
$N(X)$ & 4 & 59 & 705 & 7393 & 76807 & 787930
\end{tabular}
\end{center}
The upshot is that even though we know there are infinitely many specializations of $E'_t$ with Tamagawa number divisible by 5, it is not easily inferrable from these initial data what the explicit asymptotics might be as $X \to \infty$. See Section \ref{computations} for a more precise formulation of the statistics and computational data. In particular, for $\ell=5$ (resp. $7$), we consider the first $156$ (resp. $176$) million elliptic curves that are $\ell$-isogenous to an elliptic curve with an $\ell$-torsion point.

\subsection{Databases and Computation} There is a large computational component to this work, as seen throughout the paper, with emphasis on the final section.  All calculations were performed using Magma \cite{magma} and Sage \cite{sagemath}.  In addition, we cite many examples from the LMFDB \cite{lmfdb}.  For those examples, we provide links directly to the online database.

\section{Background and Setup} \label{background_section}

In this section we give a brief overview of the Tamagawa number of an elliptic curve.  We do not intend to be encyclopedic and refer the reader to \cite{silverman2} for detailed background.  Let $E$ be an elliptic curve defined over a global field $K$.  Let $v$ be a non-archimedean place of $K$ and denote by $K_v$ the completion of $K$ at $v$.  Let $k_v$ be the residue field at $v$ and let $E_0(K_v)$ be the points of non-singular reduction of $E(K_v)$.  Then $E_0(K_v)$ is a finite subgroup of $E(K_v)$ and the index 
\[
c_v(E/K) = [E(K_v):E_0(K_v)]
\]
is called the local Tamagawa number at $v$. The product $c(E/K) = \prod_v c_v(E/K)$ is the global Tamagawa number of $E/K$.  

We now take $K =\Q$ for the remainder of the paper.  For ease of notation, we write $c_{E,p}$ for the local Tamagawa number of $E$ at $p$ and $c_E$ for the global Tamagawa number.  Tate's algorithm is a multi-step, terminating process that determines $c_{E,p}$ and is laid out with 11 main steps in \cite[Ch.~IV.9]{silverman2} (in this paper, when we refer to Step $n$ of Tate's algorithm, we are referring to Silverman's exposition).

If $E/\Q$ is an elliptic curve with Weierstrass equation
\begin{align} \label{long_w}
y^2 + a_1xy + a_3y = x^3 + a_2x^2 + a_4x+a_6,
\end{align}
then the following invariants are well-known:
\begin{center}
\begin{tabular}{ll}
$b_2 = a_1^2 + 4a_2$ & $c_4 = b_2^2 - 24b_4$\\ 
$b_4 = 2a_4 + a_1a_3$ & $c_6 = 36b_2b_4 - 216b_6-b_2^3$ \\
$b_6 = a_3^2+4a_6$ &$\Delta =  9b_2b_4b_6 - b_2^2b_8 - 8b_4^3 - 27b_6^2$ \\
$b_8 = a_1^2a_6 + 4a_2a_6 - a_1a_3a_4 + a_2a_3^2 - a_4^2$& $j = c_4^3/\Delta$ 
\end{tabular}
\end{center}
Two elliptic curve $E$ and $E^{\prime}$ are \textit{$\Q$-isomorphic} if there exists an isomorphism $\psi:E\rightarrow E^{\prime}$ defined by $\psi(x,y)=\left(ux^{2}+r,u^{3}y+u^{2}sx+w\right)  $ for $u,r,s,w\in \Q$, with $u$ non-zero. For such an isomorphism, we write $\psi=\left[  u,r,s,w\right]  $. 

To begin calculating $c_{E,p}$, recall that $E$ has bad reduction when $p \mid \Delta$ (this is Step $1$), has multiplicative reduction at $p$ if $v_p(\Delta) >0$ and $v_p(c_4) = 0$ (this is Step $2$), and additive reduction if $v_p(\Delta) >0$ and $v_p(c_4) >0$ (these are steps $3$-$10$). If the algorithm reaches step $11$, then the model for $E$ is not minimal at $p$, and the algorithm must be run again after an admissible change of variables. A by-product of Tate's Algorithm is the Kodaira-N\'{e}ron type at a prime $p$ of an elliptic curve $E$, which we denote by $\operatorname{typ}_p(E)$. The \textit{$p$-adic signature} of an elliptic curve $E$, denoted $\operatorname{sig}_p(E)$, is the triple $(v_p(c_4),v_p(c_6),v_p(\Delta))$. In \cite{Papadopoulos1993}, it is shown that if an elliptic curve $E$ is given by a minimal Weierstrass model, then the Kodaira-N\'{e}ron type at a prime $p$ of $E$ is uniquely determined by $\operatorname{sig}_p(E)$ if $p>3$. If $p \in \lbrace 2,3 \rbrace$, then additional information is needed to determine the Kodaira-N\'{e}ron type.  All of this is summarized in the online supplement \cite{online} to \cite{br}.

If $E$ is an elliptic curve with a torsion point of order $P\geq4$, then $E$ admits a unique model over $\Q$ \cite[Proposition~1.3]{Baaziz}, the so-called Tate normal form,
\begin{equation}
y^{2}+\left(  1-g\right)  xy-fy=x^{3}-fx^{2},\label{tate}
\end{equation}
where $f\in \Q^{\times}$ and $g\in \Q$, and $P=\left(  0,0\right)  $. When $P$ has order $n=4,5,\ldots,10,12$, the corresponding modular curve $X_{1}(n)$ has genus $0$, and this leads to such curves being parameterizable by their Tate normal form with $f,g\in\mathbb{Q}(t)$ \cite{Kubert}. Now let $C_{n}$ denote the cyclic group of order $n$. For $n=4,5,\ldots,10,12$, the Tate normal form leads to a parameterized elliptic curve $E_{C_{n}}=E_{C_{n}}(a,b)$ with the property that if $E$ is an elliptic curve with $C_{n}\hookrightarrow E(\Q)$, then there are relatively prime integers $a$ and $b$ such that $E$ is $\Q$-isomorphic to $E_{C_{n}}$, and $E_{C_{n}}$ is given by an integral Weierstrass model and $P=(0,0)$ is the torsion point of order $n$ \cite[Lemma 2.9]{Barrios}. Elliptic curves with a torsion point of order $3$ are also parameterizable, but we must take special care with the elliptic curves with a $3$-torsion point and $j$-invariant $0$. As a result, we will consider two different parameterizations in this setting, namely, $E_{C_{3}}=E_{C_{3}}(a,b)$ and $E_{C_{3}^{0}}=E_{C_{3}^{0}}(a)$, where $E_{C_{3}^{0}}$ parameterizes elliptic curves with a torsion point of order $3$ and $j$-invariant $0$ \cite[Proposition 4.3]{Barrios}. For both parameterizations, we have that $P=(0,0)$ is a torsion point of order $3$. Now let $T=C_{3},C_{3}^{0},C_{5},$ or $C_{7}$. By loc. cit., every elliptic curve with a torsion point of order $\ell=3,5,7$ arises as an integral specialization of $E_{T}$. Recall that we set $\widetilde{E}_{T}=E_{T}/T$. By V\'{e}lu's explicit formulas \cite{velu}, we obtain a Weierstrass model for $\widetilde{E}_{T}$. The following proposition
summarizes this discussion.

\begin{prop}\label{isomodels}
Let $E_{T}$ and $\widetilde{E}_{T}$ be the parameterized elliptic curves given in Tables~\ref{ta:ETmodel} and~\ref{ta:ETtilmodel}, respectively. Let $E/\mathbb{Q}$ be an elliptic curve with a torsion point $P$ of order $\ell=3,5,$ or $7$, and set $\widetilde{E}=E/\langle P\rangle$. Then, there are integers $a$ and $b$ with $a$ positive such that the following hold:

\begin{enumerate}
\item[$(a)$] If $\ell=5,7$, then $E\ $and $\widetilde{E}$ are $\mathbb{Q}$-isomorphic to $E_{T}(a,b)$ and $\widetilde{E}_{T}(a,b)$, respectively, with $\gcd(a,b)=1$;

\item[$(b)$] If $\ell=3$ and the $j$-invariant of $E$ is nonzero, then $E\ $and $\widetilde{E}$ are $\mathbb{Q}$-isomorphic to $E_T(a,b)$ and $\widetilde{E}_T(a,b)$, respectively, with $\gcd(a,b)=1$;

\item[$(c)$] If $\ell=3$ and the $j$-invariant of $E$ is $0$, then $E\ $and $\widetilde{E}$ are $\mathbb{Q}$-isomorphic to either $(i)$ $E_{C_3}(24,1)$ and $\widetilde{E}_{C_3}(24,1)$, respectively, or $(ii)$ $E_{C_3^0}(a)$ and $\widetilde{E}_{C_3^0}(a)$, respectively, with $a$ cubefree.
\end{enumerate}
\end{prop}

\vspace{-0.85em} {\renewcommand*{\arraystretch}{1.18}
\begin{longtable}{cccc}
\caption[Weierstrass Model of $E_{T}$]{The Weierstrass Model $E_{T}:y^{2}+a_{1}xy+a_{3}y=x^{3}+a_{2}x^{2}$.}\\
\hline
$T$ & $a_{1}$ & $a_{2}$ & $a_{3}$  \\
\hline
\endfirsthead
\caption[]{\emph{continued}}\\
\hline
$T$ & $a_{1}$ & $a_{2}$ & $a_{3}$\\
\hline
\endhead
\hline
\multicolumn{2}{r}{\emph{continued on next page}}
\endfoot
\hline
\endlastfoot
$C_{3}^{0}$ & $0$ & $0$ & $a$ \\\hline
$C_{3}$ & $a$ & $0$ & $a^{2}b$  \\\hline
$C_{5}$ & $a-b$ & $-ab$ & $-a^{2}b$  \\\hline
$C_{7}$ & $a^{2}+ab-b^{2}$ & $a^{2}b^{2}-ab^{3}$ & $a^{4}b^{2}-a^{3}b^{3}$ 
\label{ta:ETmodel}	
\end{longtable}}

\vspace{-0.85em} {\renewcommand*{\arraystretch}{1.18}
\begin{longtable}{cC{2in}C{3.5in}}
\caption[Weierstrass Model of $\widetilde{E}_{T}$]{The Weierstrass Model $\widetilde{E}_{T}:y^{2}+a_{1}xy+a_{3}y=x^{3}+a_{2}x^{2} + a_4x +a_6$, where $a_1,a_2,a_3$ are the Weierstrass coefficients of $E_T$ in Table~\ref{ta:ETmodel}.}\\
\hline
$T$ & $a_4$ & $a_6$  \\
\hline
\endfirsthead
\caption[]{\emph{continued}}\\
\hline
$T$  & $a_4$ & $a_6$ \\
\hline
\endhead
\hline
\multicolumn{2}{r}{\emph{continued on next page}}
\endfoot
\hline
\endlastfoot
$C_{3}^{0}$ & $0$ & $-7a^2$ \\\hline
$C_{3}$ & $-5a^3b$ & $-a^4b(a+7b)$ \\\hline
$C_{5}$  & $ 5ab(a^2-2ab-b^2)$ & $ab(a^4 - 15a^3b + 5a^2b^2 -10ab^3 - b^4$ \\\hline
$C_{7}$  & $ 5ab(a-b)(a^2 - ba + b^2)(a^3 - 5a^2b + 2ab^2 + b^3)$ & $ab(a-b)(a^9 - 18a^8b + 76a^7b^2 - 182a^6b^3 + 211a^5b^4 - 132a^4b^5 +   70a^3b^6 - 37a^2b^7 + 9ab^8 + b^9)$
\label{ta:ETtilmodel}	
\end{longtable}}

\bigskip

We now turn to the main business of the paper and divide our calculations over the next sections based on $\ell \in \lbrace 3,5,7 \rbrace$.  In each case we will characterize the elliptic curves over $\Q$ that locally have a subgroup of order $\ell$ and when those Tamagawa numbers are divisible by~$\ell$.  

\section{Local Tamagawa numbers of \texorpdfstring{$p$}{p}-isogenous elliptic curves}

In this section, we explicitly classify the local Tamagawa numbers of $\ell$-isogenous rational elliptic curves $E$ and $\widetilde{E}$, where $\ell\in\left\{3,5,7\right\}  $, $E$ has an $\ell$ torsion point $P$, and $\widetilde{E}=E/\left\langle P\right\rangle $. By Proposition \ref{isomodels}, it suffices
to classify the local Tamagawa numbers of the parameterized elliptic curves $E_{T}$ and $\widetilde{E}_{T}$. This is the setting of the main result of this section, which we will use in the subsequent sections to prove our results pertaining to global Tamagawa numbers.

\begin{thm}\label{thmontama}
Let $E_{T}$ and $\widetilde{E}_{T}$ be as given in\ Tables \ref{ta:ETmodel} and \ref{ta:ETtilmodel}, respectively. Then $E_{T}$ and $\widetilde{E}_{T}$ have bad reduction at a prime $p$ if and only if their parameters satisfy one of the conditions listed in Table \ref{ta:Tamagawanumbers}, and the local Tamagawa numbers $c_{p}$ and $\widetilde{c}_{p}$ of $E_{T}$ and $\widetilde{E}_{T}$, respectively, are as given.
\end{thm}

{\begingroup \tiny
\renewcommand{\arraystretch}{1.2}
 \begin{longtable}{cccccc}
 	\caption{Local Tamagawa numbers of $E_{T}$ and $\widetilde{E}_{T}$}\\
	\hline
$T$ & $p$ & \multicolumn{2}{c}{Conditions} & $c_{p}$ & $\widetilde{c}_{p}$\\
	\hline
	\endfirsthead
	\hline
$T$ & $p$ & \multicolumn{2}{c}{Conditions} & $c_{p}$ & $\widetilde{c}_{p}$ \\
	\hline
	\endhead
	\hline 
	\multicolumn{4}{r}{\emph{continued on next page}}
	\endfoot
	\hline 
	\endlastfoot
	
$C_{3}^{0}$ & $\neq3$ & $v_{p}(a)>0$ & $p\equiv1\! \pmod{6}$ & $3$ &
$3$\\\cmidrule{4-6}
&  &  & $p\equiv2,5 \pmod{6}$ & $3$ & $1$\\\cmidrule{2-6}
& $3$ & $a\equiv\pm1 \pmod{9}$ &  & $1$ & $3$\\\cmidrule{3-6}
&  & $a\equiv\pm2 \pmod{9}$ &  & $2$ & $2$\\\cmidrule{3-6}
&  & $a\equiv\pm4 \pmod{9}$ &  & $1$ & $1$\\\cmidrule{3-6}
&  & $a\equiv0\! \pmod{3}$ &  & $3$ & $1$\\\hline
$C_{3}$ & $\ge 2$ & $n=v_{p}(b)>0$ &  & $3n$ & $n$\\\cmidrule{2-6}
& $\neq3$  & $n=v_{p}(a-27b)>0$ & $p\equiv1\! \pmod{6}$ & $n$ & $3n$\\\cmidrule{4-6}
&  &  & $p\equiv2,5\! \pmod{6}$ & $2-n \! \pmod{2}$ &
$2-n \! \pmod{2}$\\\cmidrule{3-6}
& & $v_{p}(a)\not \equiv 0\! \pmod{3}$ &
$p\equiv1\! \pmod{6}$ & $3$ & $3$\\\cmidrule{4-6}
&  &  & $p\equiv2,5\! \pmod{6}$ & $3$ & $1$\\\cmidrule{2-6}
& $3$ & $v_{3}(a-27b)=4$ &  & $1$ & $1$\\\cmidrule{3-6}

&  & $v_{3}(a) = 3,v_{3}(a-27b)=3,$ &  $\left( \frac{-b(4a+27b)/81}{3}\right) = 1$ & $1$ & $3$\\\cmidrule{4-6}
&  & $bd^{2}e^{3}\left(  b^{3}d^{2}e^{5}-c\right)  \not \equiv
7\pmod{9}$ & $\left( \frac{-b(4a+27b)/81}{3}\right) = -1$ & $1$ & $1$\\\cmidrule{3-6}

&  & $v_{3}(a)\ge 6$, $v_{3}(a)\equiv0\! \pmod{3},$ & $\left( \frac{(4-b^2d^2e^4)/3}{3}\right) = 1$  & $1$ & $3$\\\cmidrule{4-6}
&  & $ v_{3}(a-27b)=3, b^{4}d^{4}e^{8}  \not \equiv
7\pmod{9}$ & $\left( \frac{(4-b^2d^2e^4)/3}{3}\right) = -1$ & $1$ & $1$\\\cmidrule{3-6}

&  & \multicolumn{2}{c}{$v_{3}(a)\equiv0\! \pmod{3},v_{3}%
(a-27b)=3,bd^{2}e^{3}\left(  b^{3}d^{2}e^{5}-c\right)  \equiv
7\! \pmod{9}$} & $2$ & $2$\\\cmidrule{3-6}
&  & $v_{3}(a)=2$ & $\left(\frac{-bce}{3}\right)=1$ & $3$ & $3$\\\cmidrule{4-6}
&  &  & $\left(\frac{-bce}{3}\right)=-1$  & $3$ & $1$\\\cmidrule{3-6}
&  & $v_{3}(a)\equiv2\! \pmod{3},v_{3}(a)\neq2$ &  & $3$ & $1$\\\cmidrule{3-6}
&  & $v_{3}(a-27b)=5$ & $a^{2}-27ab\not \equiv 3^{8} \! \pmod{3^{9}}$
& $1$ & $1$\\\cmidrule{4-6}
&  &  & $a^{2}-27ab\equiv3^{8}\!  \pmod{3^{9}}$ & $3$ & $1$\\\cmidrule{3-6}
&  & $v_{3}(a-27b)=6$ & $a^{2}-27ab\equiv3^{9}\! \pmod{3^{10}}$ &
$1$ & $1$\\\cmidrule{4-6}
&  &  & $a^{2}-27ab\not \equiv 3^{9}\! \pmod{3^{10}}$ & $2$ & $2$\\\cmidrule{3-6}
&  & $n=v_{3}(a-27b)-6\geq1$ & $a^{2}-27ab\not \equiv 3^{9+n}%
\! \pmod{3^{10+n}}$ & $2$ & $2$\\\cmidrule{4-6}
&  &  & $a^{2}-27ab\equiv3^{9+n}\! \pmod{3^{10+n}}$ & $4$ & $4$\\\cmidrule{3-6}
&  & $v_{3}(a)\equiv1\! \pmod{3}$ &  & $3$ & $1$\\\hline
$C_{5}$ & $\geq2$ & $n=v_{p}(ab)$ &  & $5n$ & $n$\\\cmidrule{2-6}
& $\geq7$ & $n=v_{p}(a^{2}+11ab-b^{2})>0$ & $\left(  \frac{-5(a^{2}+b^{2})}%
{p}\right)  =1$ & $n$ & $5n$\\\cmidrule{4-6}
&  &  & $\left(  \frac{-5(a^{2}+b^{2})}{p}\right)  =-1$ &
$2-n \! \pmod{2}$ & $2-n \! \pmod{2}$\\\cmidrule{2-6}
& $5$ & $v_{5}(a+18b)=1$ &  & $1$ & $1$\\\cmidrule{3-6}
&  & $v_{5}(a+18b)\geq2$ &  & $2$ & $2$\\\hline
$C_{7}$ & $\geq2$ & $\,n=v_{p}(ab(a-b))>0$ &  & $7n$ & $n$\\\cmidrule{2-6}
& $\geq13$ & $n=v_{p}(a^3+5a^{2}b-$ & $\left(  \frac{-7(a^{2}-ab+b^{2})}%
{p}\right)  =1$ & $n$ & $7n$\\\cmidrule{4-6}
&  & $8ab^{2}+b^{3})>0$ & $\left(  \frac{-7(a^{2}-ab+b^{2})}{p}\right)  =-1$ &
$2-n \! \pmod{2}$ & $2-n \! \pmod{2}$\\\cmidrule{2-6}
& $7$ & $v_{7}(a+4b)\geq1$ &  & $1$ & $1$

\label{ta:Tamagawanumbers}
\end{longtable}
\endgroup}

\begin{proof}
Since $E_{T}$ and $\widetilde{E}_{T}$ are isogenous, they have the same reduction type at each prime. Theorems 3.4, 3.5, and 3.8 of \cite{br} give necessary and sufficient conditions on the parameters of $E_{T}$ to determine the primes $p$ at which $E_{T}$ has additive reduction, and the local Tamagawa number $c_{p}$. Similarly, Theorem 2.2 of \cite{br2} provides necessary and sufficient conditions on the parameters of $E_{T}$ to determine the prime $p$ at which $E_{T}$ has split or non-split multiplicative reduction. The local Tamagawa number $c_{p}$ then follows since if $E_{T}$ has split multiplicative reduction, then $c_{p}=v_{p}(\Delta_{T})$. Similarly, if $E_{T}$ has non-split multiplicative reduction, then $c_{p}=2-(v_{p}(\Delta_{T})\! \pmod{2})$. It is now verified from \cite{br,br2} that the conditions listed in Table~\ref{ta:Tamagawanumbers} are exactly the cases for which $E_{T}$ has bad reduction at a prime $p$, and that the local Tamagawa number $c_{p}$ is as claimed. It suffices to show that $\widetilde{c}_{p}$ is as claimed in Table \ref{ta:Tamagawanumbers}. We proceed by cases. We note that cases~$2$ and~$3$ are involved, and the reader can find an accompaniment to this proof in \cite[Theorem3\_1.ipynb]{GitHubIsog}. This file contains code written in SageMath, which verifies computational claims made in the proof.

\textbf{Case 1.} Suppose $T=C_{3}^{0}$, and let $a$ be a cubefree positive integer. Since $E_{T}$ has $j$-invariant $0$, we have that if $E_{T}$ has bad reduction at a prime $p$, then $E_{T}$ has additive potentially good reduction at $p$. By \cite[Theorem 3.4]{br}, $E_{T}$ has additive potentially good reduction at a prime $p$ if and only if $p|3a$. If $p\neq3$, then $\operatorname*{typ}_{p}(E_{T})\in\left\{  \rm{IV},\rm{IV}^{\ast}\right\}  $ by loc. cit. with $c_{p}=3$. The result now follows since $E_{T}$ has a $3$-torsion point and \cite[Theorem 6.1]{Dokchitsersquared} gives
\[
\frac{c_{p}}{\widetilde{c}_{p}}=\left\{
\begin{array}
[c]{cl}%
1 & \text{if }\left(  \frac{-3}{p}\right)  =1,\\
3 & \text{if }\left(  \frac{-3}{p}\right)  =-1.
\end{array}
\right.
\]
It remains to consider the case when $p=3$. We now compute $\widetilde{c}_{3}$ by cases.

\qquad\textbf{Subcase 1a.} Suppose $v_{3}(a)=0$. Let $\widetilde{F}_{T}$ be the elliptic curve obtained from $\widetilde{E}_{T}$ via the isomorphism $\left[  1,3,0,\frac{-a}{2}\right]  $. Then
\[
\widetilde{F}_{T}:y^{2}=x^{3}+9x^{2}+27x-\frac{27}{4}\left(  a^{2}-4\right)  .
\]
Note that
\[
a^{2}-4\equiv\left\{
\begin{array}
[c]{rl}%
\pm3\! \pmod{9} & \text{if }a\equiv\pm1,\pm4\! \pmod{9},\\
0\! \pmod{9} & \text{if }a\equiv\pm2\! \pmod{9}.
\end{array}
\right.
\]
It is now checked using Tate's Algorithm that%
\[
\operatorname*{typ}\nolimits_{3}(\widetilde{F}_{T})=\left\{
\begin{array}
[c]{cl}%
\rm{IV}^{\ast} & \text{if }a\equiv\pm1,\pm4\! \pmod{9},\\
\rm{III}^{\ast} & \text{if }a\equiv\pm2\! \pmod{9}.
\end{array}
\right.
\]
In particular, if $a\equiv\pm2\! \pmod{9}$, then $\widetilde{c}_{3}=2$. Now suppose $a\equiv\pm1,4\! \pmod{9}$. Then $\operatorname*{typ}\nolimits_{3}(\widetilde{F}_{T})=\rm{IV}^{\ast}$, and by Tate's Algorithm, $\widetilde{c}_{3}$ is determined by whether the following polynomial splits in $\mathbb{F}_{3}$:
\[
Y^{2}+\frac{1}{12}\left(  a^{2}-4\right)  \equiv\left\{
\begin{array}
[c]{cl}%
Y^{2}+2\! \pmod{3} & \text{if }a\equiv\pm1\! \pmod{9},\\
Y^{2}+1\! \pmod{3} & \text{if }a\equiv\pm4\! \pmod{9}.
\end{array}
\right.
\]
It follows that%
\[
\widetilde{c}_{3}=\left\{
\begin{array}
[c]{cl}%
3 & \text{if }a\equiv\pm1\! \pmod{9},\\
1 & \text{if }a\equiv\pm4\! \pmod{9}.
\end{array}
\right.
\]

\qquad\textbf{Subcase 1b.} Suppose $v_{3}(a)=1,2$. Then
\[
\operatorname*{sig}\nolimits_{3}(\widetilde{E}_{T})=\left(  \infty
,6+2v_{3}(a),9+4v_{3}(a)\right)  =\left\{
\begin{array}
[c]{cl}%
\left(  \infty,8,13\right)   & \text{if }v_{3}(a)=1,\\
\left(  \infty,10,17\right)   & \text{if }v_{3}(a)=2.
\end{array}
\right.
\]
If $v_{3}(a)=1$, we conclude by \cite[Tableau II]{Pap} that
$\operatorname*{typ}_{3}(\widetilde{E}_{T})=\rm{II}^{\ast}$. Now suppose $v_{3}(a)=2$. By \cite[Tableau II]{Pap}, $\widetilde{E}_{T}$ is not a minimal model at $3$. Therefore, there exists a minimal model $\widetilde{E}_{T}^{\prime}$ at $3$ for $\widetilde{E}_{T}$ that satisfies $\operatorname*{sig}_{3}(\widetilde{E}_{T}^{\prime})=\left(  \infty ,4,5\right)  $. By \cite[Tableau II]{Pap}, we conclude that $\operatorname*{typ}_{3}(\widetilde{E}_{T})=\rm{II}$. Consequently, in both cases, we have that $\widetilde{c}_{3}=1$.

\textbf{Case 2.} Suppose $T=C_{3}$. Let $a=c^{3}d^{2}e$ for positive integers $c,d,e$ with $de$ squarefree. Let $F_{T}$ and $\widetilde{F}_{T}$ be the elliptic curves obtained from $E_{T}$ and $\widetilde{E}_{T}$, respectively, via the isomorphism $\left[  c^{2}d,0,0,0\right]  $. By \cite[Theorem 6.1]{Barrios}, $F_{T}$ is a global minimal model for $E_{T}$. The elliptic curve $\widetilde{F}_{T}$ is given by an integral Weierstrass
model and
\[
\operatorname*{sig}(\widetilde{F}_{T})=\left(  cd^{2}e^{3}\left(
a+216b\right)  ,-d^{2}e^{4}\left(  a^{2}-540ab-5832b^{2}\right)  ,d^{4}%
e^{8}b\left(  a-27b\right)  ^{3}\right)  .
\]
We now consider the cases in Table \ref{ta:Tamagawanumbers} separately.

\qquad\textbf{Subcase 2a.} Let $p$ be a prime and suppose that $n=v_{p}(b)>0$. By \cite[Theorem 2.2]{br2}, $E_{T}$ has split multiplicative reduction and $v_{p}(\Delta_{T})=3n$. Since $v_{p}(\widetilde{\Delta}_{T})=n$, it follows that $\widetilde{c}_{p}=n$.

\qquad\textbf{Subcase 2b.} Let $p\neq3$ be a prime and suppose that $n=v_{p}(a-27b)>0$. By \cite[Theorem 2.2]{br2}, $E_{T}$ has multiplicative reduction and $v_{p}(\Delta_{T})=n$. Moreover, $E_{T}$ has split multiplicative reduction if and only if $p\equiv1\! \pmod{6}$. The result follows since $v_{p}(\widetilde{\Delta}_{T})=3n$.

\qquad\textbf{Subcase 2c. }Let $p\neq3$ be a prime and suppose that $v_{p}(a)\not \equiv 0\! \pmod{3}$. By \cite[Theorem 3.5]{br}, $\operatorname*{typ}_{p}(E_{T})\in\left\{  \rm{IV},\rm{IV}^{\ast}\right\}  $ with $c_{p}=3$. The result now follows since $E_{T}$ has a $3$-torsion point and \cite[Theorem 6.1]{Dokchitsersquared} gives
\[
\frac{c_{p}}{\widetilde{c}_{p}}=\left\{
\begin{array}
[c]{cl}%
1 & \text{if }\left(  \frac{-3}{p}\right)  =1,\\
3 & \text{if }\left(  \frac{-3}{p}\right)  =-1.
\end{array}
\right.
\]

\qquad\textbf{Subcase 2d.} Suppose $v_{3}(a-27b)=4$, so that $v_{3}(a)=3$. By \cite[Theorem 3.5]{br}, $\operatorname*{typ}_{3}(E_{T})=\rm{II}$ with $c_{3}=1$. Under these assumptions, it is checked that $\operatorname*{sig}\nolimits_{3}(\widetilde{F}_{T})=\left(  5,8,12\right)  $. By \cite[Tableau~II]{Pap}, $\operatorname*{typ}_{3}(\widetilde{E}_{T})=\rm{II}^{\ast}$. It follows that $\widetilde{c}_{3}=1$.

\qquad\textbf{Subcase 2e.} Suppose $v_{3}(a)=3$, $v_{3}(a-27b)=3,$ and $bd^{2}e^{3}\left(  b^{3}d^{2}e^{5}-c\right)  \not \equiv 7\! \pmod{9}$. In particular, $v_{3}(c)=1$ and $v_{3}(de)=0$. By \cite[Theorem 3.5]{br}, $\operatorname*{typ}_{3}(E_{T})=\rm{II}$ with $c_{3}=1$. Let $\widetilde{F}_{T}^{\prime}$ denote the elliptic curve obtained from $\widetilde{F}_{T}$ via the isomorphism $\left[  1,0,\frac{-cde}{2},\frac{-bde^{2}}{2}\right]  $. Then
\begin{equation}
\widetilde{F}_{T}^{\prime}:y^{2}=x^{3}+\frac{1}{4}c^{2}d^{2}e^{2}x^{2}-\frac{9}{2}bcd^{2}e^{3}x-\frac{1}{4}bd^{2}e^{4}\left(  4a+27b\right)  .
\label{modelFTtil}%
\end{equation}
In \cite{GitHubIsog}, we verify that $v_{3}(4a+27b)=4$. It is then checked that $\widetilde{F}_{T}^{\prime}$ satisfies the first eight steps of Tate's Algorithm. 
The conclusion now follows, since step $8$ of Tate's Algorithm allows us to conclude that $\operatorname*{typ}_{3}(\widetilde{F}_{T}^{\prime})=\rm{IV}^{\ast}$ and
\[
\widetilde{c}_{3}=\left\{
\begin{array}
[c]{cl}%
3 & \text{if }\left(  \frac{-\frac{1}{4}bd^{2}e^{4}\left(  4a+27b\right)
/81}{3}\right)  =1,\\
1 & \text{if }\left(  \frac{-\frac{1}{4}bd^{2}e^{4}\left(  4a+27b\right)
/81}{3}\right)  =-1.
\end{array}
\right.
\]

\qquad\textbf{Subcase 2f.} Suppose $v_{3}(a)\geq6$, $v_{3}(a)\equiv 0\! \pmod{3}$, $bd^{2}e^{3}\left(  b^{3}d^{2}e^{5}-c\right)  \not \equiv 7\! \pmod{9}$, and $v_{3}(a-27b)=3$. In particular, $v_{3}(c)\geq2$ and $v_{3}(de)=0$. By \cite[Theorem 3.5]{br}, $\operatorname*{typ}_{3}(E_{T})=\rm{II}$ with $c_{3}=1$. Next, let $\widetilde{F}_{T}^{\prime}$ denote the elliptic curve obtained from $\widetilde{F}_{T}$ via the isomorphism $\left[  1,3-\frac{1}{3}c^{2}d^{2}e^{2},\frac{-1}{2} cde,\frac{de}{6}\left(  a-3be-9c\right)  \right]  $. The Weierstrass coefficients $a_{i}^{\prime}$ of $\widetilde{F}_{T}^{\prime}$ are
\begin{align*}
a_{1}^{\prime}  &  =a_{3}^{\prime}=0,\ a_{2}^{\prime}=9-\frac{3}{4}c^{2}d^{2}e^{2},\ a_{4}^{\prime}=\frac{1}{6}cd^{2}e^{3}\left(  a-27b\right) -\frac{9}{2}c^{2}d^{2}e^{2}+27,\\
a_{6}^{\prime}  &  =\frac{-d^{2}e^{2}}{108}\left(  ae-27c\right)  \left(ae-54be-27c\right)  +\frac{27}{4}\left(  4-b^{2}d^{2}e^{4}\right)  .
\end{align*}
By inspection, we see that $v_{3}(a_{2}^{\prime})=2,v_{3}(a_{4}^{\prime})\geq4,$ and $v_{3}(a_{6}^{\prime})=4$. For the last equality, \cite{GitHubIsog} verifies that $v_{3}(4-b^{2}d^{2}e^{4})=1$. It follows from Tate's Algorithm that $\operatorname*{typ}_{3}(\widetilde{F}_{T}^{\prime})=\rm{IV}^{\ast}$ and
\[
\widetilde{c}_{3}=\left\{
\begin{array}
[c]{cl}
3 & \text{if }\left(  \frac{(4-b^{2}d^{2}e^{4})/3}{3}\right)  =1,\\
1 & \text{if }\left(  \frac{(4-b^{2}d^{2}e^{4})/3}{3}\right)  =-1.
\end{array}
\right.
\]

\qquad\textbf{Subcase 2g.} Suppose $v_{3}(a)\equiv0\! \pmod{3}$, $v_{3}(a-27b)=3,$ and $bd^{2}e^{3}\left(  b^{3}d^{2}e^{5}-c\right) \equiv7\! \pmod{9}$. In particular, $v_{3}(c)>0$ and $v_{3}(de)=0$. By \cite[Theorem 3.5]{br}, $\operatorname*{typ}_{3}(E_{T})=\rm{III}$ with $c_{3}=2$. By \cite[Theorem 3.1]{Dokchitsersquared}, $E_{T}$ has potentially supersingular tame reduction. It then follows from \cite[Theorem~5.4]{Dokchitsersquared} that $\operatorname*{typ}_{3}(\widetilde{E}_{T})=\rm{III}^{\ast}$. Consequently, $\widetilde{c}_{3}=2$.

\qquad\textbf{Subcase 2h.} Suppose $v_{3}(a)=2$. Then $v_{3}(d)=1$ and $v_{3}(ce)=0$. By \cite[Theorem 3.5]{br}, $\operatorname*{typ}_{3}(E_{T})=\rm{IV}$ with $c_{3}=3$. Let $\widetilde{F}_{T}^{\prime}$ denote the elliptic curve obtained from $\widetilde{F}_{T}$ via the isomorphism $\left[  1,0,\frac{-cde}{2},\frac{-bde^{2}}{2}\right]  $. Then, the Weierstrass model of $\widetilde{F}_{T}^{\prime}$ is as given in (\ref{modelFTtil}). It is then checked that $\widetilde{F}_{T}$ satisfies the first eight steps of Tate's Algorithm and that $\operatorname*{typ}_{3}(\widetilde{F}_{T})=\rm{IV}^{\ast}$. It follows that the local Tamagawa number depends on the splitting of the following polynomial in $\mathbb{F}_{3}$:
\[
Y^{2}+\frac{1}{4\cdot81}bd^{2}e^{4}\left(  4a+27b\right)  \equiv Y^{2}+\frac{ab}{9}\! \pmod{3}=Y^{2}+bce\! \pmod{3}.
\]

\qquad\textbf{Subcase 2i.} Suppose $v_{3}(a)\equiv2\! \pmod{3}$ and $v_{3}(a)\neq2$. By \cite[Theorem 3.5]{br}, $\operatorname*{typ}_{3}(E_{T})=\rm{IV}$ with $c_{3}=3$. Moreover, the assumptions on $a$ imply that $v_{3}(a)\geq5,\ v_{3}(c)\geq1,$ and $v_{3}(d)=1$. It is now checked that $\operatorname*{sig}\nolimits_{3}(\widetilde{F}_{T})=\left(  \geq 6,8,13\right)  $. By \cite[Tableau II]{Pap}, $\operatorname*{typ}_{3}(\widetilde{E}_{T})=\rm{II}^{\ast}$, and thus $\widetilde{c}_{3}=1$. 

\qquad\textbf{Subcase 2j.} Suppose $v_{3}(a-27b)=5$. It follows from \cite[Theorem 3.5]{br} that $\operatorname*{typ}_{3}(E_{T})=\rm{IV}$, and that $c_{3}$ is as claimed in Table \ref{ta:Tamagawanumbers}. Next, observe that under the current assumption, $v_{3}(a)=3$ and $v_{3}(deb)=0$. With these assumptions, it is verified that $\operatorname*{sig}\nolimits_{3}(\widetilde{F}_{T})=\left(  \geq6,9,15\right)  $. By \cite[Tableau II]{Pap}, $\widetilde{F}_{T}$ is not a minimal model at $3$. Consequently, a minimal model $\widetilde{F}_{T}^{\prime}$ of $\widetilde{F}_{T}$ at $3$ must satisfy $\operatorname*{sig}\nolimits_{3}(\widetilde{F}_{T})=\left(  \geq2,3,3\right)$. By \cite[Tableau II]{Pap},  $\operatorname*{typ}_{3}(\widetilde{E}_{T})\in\left\{  \rm{II},\rm{III}\right\}  $. By \cite[Theorem 3.5]{br}, the conductor exponent $f_{3}$ at $3$ of $E_{T}$ is $3$. Since the conductor is an isogeny invariant, we conclude from \cite[Tableau II]{Pap} that $\operatorname*{typ}_{3}(\widetilde{E}_{T})=\rm{II}$. Thus, $\widetilde{c}_{3}=1$.

\qquad\textbf{Subcase 2k.} Suppose $v_{3}(a-27b)\geq6$. By \cite[Theorem~3.5]{br}, $\operatorname*{typ}_{3}(E_{T})=\rm{I}_{n}^{\ast}$ where $n=v_{3}(a-27b)-6$. Moreover, $c_{3}$ is as claimed in Table \ref{ta:Tamagawanumbers}. We conclude from \cite[Theorem~6.1]{Dokchitsersquared} that $c_{3}=\widetilde{c}_{3}$.

\qquad\textbf{Subcase 2l.} Suppose $v_{3}(a)\equiv1\! \pmod{3}$. Then $v_{3}(e)=1$ and $v_{3}(bd)=0$. It is then verified that
\[
\operatorname*{sig}\nolimits_{3}(\widetilde{F}_{T})=\left\{
\begin{array}
[c]{cl}%
\left(  4,6,11\right)   & \text{if }v_{3}(a)=1,\\
\left(  \geq7,10,17\right)   & \text{if }v_{3}(a)\geq4.
\end{array}
\right.
\]
In particular, if $v_{3}(a)=1$, then $\operatorname*{typ}_{3}(\widetilde{E}_{T})=\rm{II}^{\ast}$ by \cite[Tableau~II]{Pap}. Hence, $\widetilde{c}_{3}=1$. So suppose that $v_{3}(a)\geq4$. By \cite[Tableau~II]{Pap}, $\widetilde{F}_{T}$ is not a minimal model at $3$. Therefore, a minimal model $\widetilde{F}_{T}^{\prime}$ of $\widetilde{F}_{T}$ at $3$ must satisfy $\operatorname*{sig}\nolimits_{3}(\widetilde{F}_{T})=\left(  \geq3,4,5\right)  $. By \cite[Tableau~II]{Pap}, $\operatorname*{typ}_{3}(\widetilde{E}_{T})=\rm{II}$, and thus $\widetilde{c}_{3}=1$.

\textbf{Case 3.} Suppose $T=C_{5}$. By \cite[Theorem 6.1]{Barrios}, $E_{T}$ is given by a global minimal model. We now consider the cases in Table \ref{ta:Tamagawanumbers} separately.

\qquad\textbf{Subcase 3a.} Let $p$ be a prime and $n=v_{p}(ab)>0$. By \cite[Theorem 2.2]{br2}, $E_{T}$ has split multiplicative reduction and $v_{p}(\Delta_{T})=5n$. Since $v_{p}(\widetilde{\Delta}_{T})=n$, it follows that $\widetilde{c}_{p}=n$.

\qquad\textbf{Subcase 3b.} Let $p\geq7$ be a prime such that $n=v_{p}(a^{2}+11ab-b^{2})>0$. Then $v_{p}(\Delta_{T})=n$ and by \cite[Theorem~2.2]{br2}, $E_{T}$ has split multiplicative reduction if and only if
\[
\left(  \frac{-5\left(  a^{2}+b^{2}\right)  }{p}\right)  =1.
\]
Since $v_{p}(\widetilde{\Delta}_{T})=5n$, it follows that the local Tamagawa number $\widetilde{c}_{p}$ is as claimed.

\qquad\textbf{Subcase 3c}. Suppose $E_{T}$ has additive reduction at a prime $p$. By \cite[Theorem 3.8]{br}, this is equivalent to $p=5$ and $v_{5}(a+18b)\geq1$. Moreover, the Kodaira-N\'{e}ron type of $E_{T}$ is $\rm{II}$ (resp. $\rm{III}$) if $v_{5}(a+18)=1$ (resp. $\geq2$). By \cite[Theorems~3.1 and~5.4]{Dokchitsersquared}, the Kodaira-N\'{e}ron type of $\widetilde{E}_{T}$ is $\rm{II}^{\ast}$ (resp. $\rm{III}^{\ast}$) . Thus, $\widetilde{c}_{5}=1$ (resp. $2$).

\textbf{Case 4.} Suppose $T=C_{7}$. By \cite[Theorem 6.1]{Barrios}, $E_{T}$ is given by a global minimal model. We now consider the cases in Table \ref{ta:Tamagawanumbers} separately.

\qquad\textbf{Subcase 4a.} Let $p$ be a prime and $n=v_{p}(ab(a-b))>0$. By \cite[Theorem 2.2]{br2}, $E_{T}$ has split multiplicative reduction and $v_{p}(\Delta_{T})=7n$. Since $v_{p}(\widetilde{\Delta}_{T})=n$, it follows that $\widetilde{c}_{p}=n$.

\qquad\textbf{Subcase 4b.} Let $p\geq13$ be a prime and $n=v_{p}(a^{3}+5a^{2}b-8ab^{2}+b^{3})>0$. Then $v_{p}(\Delta_{T})=n$ and by \cite[Theorem~2.2]{br2}, $E_{T}$ has split multiplicative reduction if and only if
\[
\left(  \frac{-7\left(  a^{2}-ab+b^{2}\right)  }{p}\right)  =1.
\]
Since $v_{p}(\widetilde{\Delta}_{T})=7n$, it follows that the local Tamagawa number $\widetilde{c}_{p}$ is as claimed.

\qquad\textbf{Subcase 4c.} Suppose $E_{T}$ has additive reduction at a prime $p$. By \cite[Theorem 3.8]{br}, this is equivalent to $p=7$ and $v_{7}(a+4b)\geq1$. In addition, the Kodaira-N\'{e}ron type of $E_{T}$ is $\rm{II}$. By \cite[Theorems~3.1 and~5.4]{Dokchitsersquared}, the Kodaira-N\'{e}ron type of $\widetilde{E}_{T}$ is $\rm{II}^{\ast}$. It follows that $\widetilde{c}_{7}=1$.
\end{proof}

\section{Tamagawa Divisibility by  \texorpdfstring{$3$}{3}} \label{3sect}

This section investigates the divisibility of the global Tamagawa number of $3$-isogenous rational elliptic curves $E$ and $\widetilde{E}$, where $E$ has a $3$-torsion point $P$ and $\widetilde{E}=E/\left\langle P\right\rangle $. We begin by first restricting to the case when $E$ has $j$-invariant $0$, and then consider the general case when the $j$-invariant is nonzero. Indeed, by Proposition \ref{isomodels}, these two cases correspond to studying the parameterized elliptic curves $E_{C_{3}}$ and $E_{C_{3}^{0}}$, respectively. 

\begin{prop}
\label{PropTama3j0}There are infinitely many elliptic curves $E/\mathbb{Q}$ with $j$-invariant $0$ and a $3$-torsion point $P$ such that one of the following holds:

\begin{enumerate}
\item The global Tamagawa number $\widetilde{c}$ of $\widetilde{E}=E/\left\langle P\right\rangle $ is divisible by $3$;

\item The global Tamagawa number $\widetilde{c}$ of $\widetilde{E}=E/\left\langle P\right\rangle $ is $1$.
\end{enumerate}
\end{prop}

\begin{proof}
By Proposition \ref{isomodels}, if $E$ has $j$-invariant $0$ and a $3$-torsion point $P$, then $E$ and $\widetilde{E}=E/\left\langle P\right\rangle $ are $\mathbb{Q}$-isomorphic to either $\left(  i\right)  $ $E_{C_{3}}(24,1)$ and $\widetilde{E}_{C_{3}}(24,1)$, respectively, or $\left(  ii\right)  $ $E_{C_{3}^{0}}(a)$ and $\widetilde{E}_{C_{3}^{0}}(a)$, respectively, for some positive cubefree integer $a$. Thus, to prove the proposition, it suffices to consider the parameterized family of elliptic curves $\widetilde{E}_{C_{3}^{0}}(a)$. We now proceed by cases.

\textbf{Case 1.} Let $p\equiv1\! \pmod{6}$ be a prime. Then for each cubefree positive integer $a$ that is divisible by $p$, it is the case that $3$ divides the global Tamagawa number $\widetilde{c}$ of $\widetilde{E}_{C_{3}^{0}}(a)$ by Theorem \ref{thmontama}.

\textbf{Case 2.} Let $3a$ be a positive cubefree integer such that for each prime $p\neq3$ dividing $a$, it is the case that $p\equiv2,5\! \pmod{6}$. By Theorem \ref{thmontama}, we conclude that the global Tamagawa number $\widetilde{c}$ of $\widetilde{E}_{C_{3}^{0}}(a)$ is $1$.
\end{proof}

Note that the condition that $p\equiv1\! \pmod{6}$ is equivalent to $\left(  \frac{-3}{p}\right)  =1$. In particular, half of the primes satisfy this condition. Our next results builds on Proposition \ref{PropTama3j0} to provide a density result on the primes $p$ for which the global Tamagawa number of $\widetilde{E}_{C_{3}^{0}}(p)$ is divisible by $3$.

\begin{cor}
For a prime $p$, consider the elliptic curve $\widetilde{E}_{C_{3}^{0}}(p)$. Then,

\begin{enumerate}
\item The global Tamagawa number of $\widetilde{E}_{C_{3}^{0}}(p)$ is divisible by $3$ for $\frac{2}{3}$ of the primes;

\item The global Tamagawa number of $\widetilde{E}_{C_{3}^{0}}(p)$ is not divisible by $3$ for $\frac{1}{3}$ of the primes.
\end{enumerate}
\end{cor}

\begin{proof}
By Theorem \ref{thmontama}, the global Tamagawa number $\widetilde{c}$ of $\widetilde{E}_{C_{3}^{0}}(p)$ is $\widetilde{c}=\widetilde{c}_{3}\widetilde{c}_{p}$. Note that the condition that $p\equiv1\! \pmod{6}$ is equivalent to $\left(  \frac{-3}{p}\right)  =1$. In particular, half of the primes satisfy this condition. For these primes $p$, Theorem \ref{thmontama} implies that $\widetilde{c}$ is divisible by $3$.

Now suppose that $p\equiv2,5\! \pmod{6}$, or equivalently $\left(\frac{-3}{p}\right)  =-1$. Half the primes satisfy this condition, and we have that $\widetilde{c}_{p}=1$. However, $\widetilde{c}_{3}=3$ if and only if $p\equiv\pm1\! \pmod{9}$. Equivalently, $p\equiv -1\! \pmod{9}$ since $p\equiv2,5\! \pmod{6}$. In particular, $\frac{1}{6}$ of the primes satisfy this condition. We conclude that $\frac{2}{3}$ of the primes satisfy the condition that $\widetilde{c}$ is divisible by $3$, and that for $\frac{1}{3}$ of the prime, $\widetilde{c}$ is coprime to $3$.
\end{proof}

By \cite[Theorem 1.2]{br}, there is exactly one elliptic curve with $j$-invariant $0$ and a $3$-torsion point that has global Tamagawa number equal to $1$. Namely, $E_{C_{3}^{0}}(1):y^{2}+y=x^{3}$. Our next result shows that this is the only elliptic curve with $j$-invariant $0$ and a $3$-torsion point~$P$ such that its $3$-isogenous elliptic curve $E/\left\langle P\right\rangle $ also has a $3$-torsion point. Equivalently, $\operatorname*{im}\overline{\rho}_{E_{C_{3}^{0},3}(1)}$ is a split Cartan subgroup of $\operatorname*{GL}\nolimits_{2}(\mathbb{F}_{3})$. 

\begin{prop}\label{j0-3-torsiso}
Let $E/\mathbb{Q}$ be an elliptic curve with $j$-invariant $0$ and a $3$-torsion point~$P$. Then $\widetilde{E}=E/\left\langle P\right\rangle $ has a $3$-torsion point if and only if $E$ and $\widetilde{E}$ are $\mathbb{Q}$-isomorphic to $E_{C_{3}^{0}}(1)$ and $\widetilde{E}_{C_{3}^{0}}(1)$, respectively. Moreover, $\widetilde{E}_{C_{3}^{0}}(1)$ is $\mathbb{Q}$-isomorphic to $E_{C_{3}}(24,1)$, which is the elliptic curve with LMFDB label \href{https://www.lmfdb.org/EllipticCurve/Q/27/a/3}{27.a3}.
\end{prop}

\begin{proof}
By Proposition \ref{isomodels}, if $E$ has $j$-invariant $0$ and a $3$-torsion point $P$, then $E$ and $\widetilde{E}=E/\left\langle P\right\rangle $ are $\mathbb{Q}$-isomorphic to either $\left(  i\right)  $ $E_{C_{3}}(24,1)$ and $\widetilde{E}_{C_{3}}(24,1)$, respectively, or $\left(  ii\right)  $ $E_{C_{3}^{0}}(a)$ and $\widetilde{E}_{C_{3}^{0}}(a)$, respectively, for some positive cubefree integer $a$. It is verified that the torsion subgroup of $\widetilde{E}_{C_{3}}(24,1)$ is trivial. So suppose that $E$ is $\mathbb{Q}$-isomorphic to $E_{C_{3}^{0}}(a)$. The $3$-division polynomial of $\widetilde{E}_{C_{3}^{0}}(a)$ shows that the $x$-coordinates of the $3$-torsion points of $\widetilde{E}_{C_{3}^{0}}(a)$ satisfy%
\[
x(x^{3}-27a^{2})=0.
\]
If $x=0$, then the $y$-coordinates are roots of $y^{2}+ay-7a^{2}$, which is irreducible over $\mathbb{Q}$ for all values of $a$. The polynomial $x^{3}-27a^{2}$ has an integral root if and only if $a$ is a cube. Since $a$ is assumed to be cube-free, we are reduced to the case $a=1$, in which we verify that the torsion subgroup of $\widetilde{E}_{C_{3}^{0}}(1)$ is $\left\{  \infty,\left(  3,4\right) ,\left(  3,-5\right)  \right\}  $. Lastly, there is a $\mathbb{Q}$-isomorphism from $\widetilde{E}_{C_{3}^{0}}(1)$ to $E_{C_{3}}(24,1)$ given by $\left[  \frac{1}{4},3,3,4\right]  $. Moreover, this elliptic curve with LMFDB label~\href{https://www.lmfdb.org/EllipticCurve/Q/27/a/3}{27.a3}.

Next, we consider the more general case of an elliptic curve with $3$-torsion, whose $j$-invariant is non-zero. Similar to the $j$-invariant $0$ case, we begin by considering those elliptic curves whose $3$-isogenous elliptic curve also has a $3$-torsion point.
\end{proof}

\begin{prop} \label{3torsprop}
Let $E/\mathbb{Q}$ be an elliptic curve with non-zero $j$-invariant and a $3$-torsion point $P$. If $\widetilde{E}=E/\left\langle P\right\rangle $ has a $3$-torsion point, then there are relatively prime integers $a,b$ that are cubes with $a>0$ such that $E$ and $\widetilde{E}$ are $\mathbb{Q}$-isomorphic to $E_{C_{3}}(a,b)$ and $\widetilde{E}_{C_{3}}(a,b)$, respectively.
\end{prop}

\begin{proof}
Let $T=C_{3}$. By Proposition \ref{isomodels}, there are relatively prime integers $a$ and $b$ with $a>0$ such that $E$ and $\widetilde{E}$ are $\mathbb{Q}$-isomorphic to $E_{T}(a,b)$ and $\widetilde{E}_{T}(a,b)$, respectively. It remains to show that $a,b$ are cubes. To this end, we start by computing the 3-division polynomial of $\widetilde{E}_{T}(a,b)$.  A calculation reveals that the $x$-coordinates of the non-trivial 3-torsion points are roots of the quartic
\[
(3x + a^2)(x^3 - 9a^3bx - a^4b(a-27b)).
\]
If $x=-a^2/3$, then the $y$-coordinates are the roots of the quadratic
\[
y^2 + ((-1/3)a^3 + ba^2)y + ((1/27)a^6 - (2/3)ba^5 + 7b^2a^4),
\]
which has discriminant $(-1/3) \left(a^2(a-27b)/3\right)^2 \not \in \Q^{\times 2}$.  Therefore, $\widetilde{E}_{T}(a,b)$ will have a point of order 3 only if 
\begin{align} \label{cub}
x^3 - 9a^3bx - a^4b(a-27b)
\end{align}
has a rational root. We now appeal to known  results that tell us when a cubic has a rational root.

Following \cite{schulz}, we set $p(a,b) =  - 9a^3b$, $q(a,b) = - a^4b(a-27b)$, and 
\begin{align*}
D(a,b) &= \left(\frac{p(a,b)}{3}\right)^3 + \left(\frac{q(a,b)}{2} \right)^2 = \left(\frac{ba^4(a-27b)}{2}\right)^2
\end{align*}
and choose $\sqrt{D(a,b)} = \frac{ba^4(a-27b)}{2}$.  Set
\begin{align*}
H(a,b) &= \sqrt[3]{-q(a,b)/2 +\sqrt{D(a,b)}} = a \sqrt[3]{ba^2} \\
K(a,b) &=\sqrt[3]{-q(a,b)/2 -\sqrt{D(a,b)}} = 3a\sqrt[3]{b^2a}.
\end{align*}

With this notation in place, the roots of the polynomial (\ref{cub}) are 
\begin{align*}
x_1(a,b) &\ddef H(a,b) + K(a,b) \\
x_2(a,b) &\ddef \omega H(a,b) + \omega^2 K(a,b) \\
x_3(a,b) &\ddef \omega^2H(a,b) + \omega K(a,b),
\end{align*}
where $\omega$ is a primitive cube root of unity \cite[(2)]{schulz}.  Then $(\ref{cub})$ has a rational root if and only if $x_1(a,b)$ is rational.  Observe that $H(a,b)$ is rational if and only if $K(a,b)$ is rational if and only if their sum is rational, which is true if and only if \emph{both} $a$ and $b$ are rational cubes, since they are coprime. 

Set $a=s^3$ and $b=t^3$ so that 
\[
x_1(a,b) = s^4t(s+3t).
\]
Substituting this value into $\widetilde{E}_T(s^3,t^3)$ yields
\begin{align*}
(y - 4s^6t^3)(y + s^8t + 3s^7t^2 + 5s^6t^3)=0.
\end{align*}
We conclude from this that the point $(s^4t(s+3t),4s^6t^3)$ is a rational point of order 3.
\end{proof}

We can now use this information to show that there are infinitely many specializations of $\widetilde{E}_{T}(a,b)$ that have Tamagawa number divisible by 3.

\begin{prop}
\label{PropTama3}There are infinitely many elliptic curves $E/\mathbb{Q}$ with non-zero $j$-invariant and a $3$-torsion point $P$ such that one of the
following holds:

\begin{enumerate}
\item The global Tamagawa number $\widetilde{c}$ of $\widetilde{E}=E/\left\langle P\right\rangle $ is divisible by $3$;

\item The global Tamagawa number $\widetilde{c}$ of $\widetilde{E}=E/\left\langle P\right\rangle $ is $1$.
\end{enumerate}
\end{prop}

\begin{proof}
Let $T=C_{3}$, and observe that for each odd prime $p$, Theorem~\ref{thmontama} implies that $\operatorname*{typ}_{p}(\widetilde{E}_{T}(2,p^{3}))=\rm{I}_{3}$ with $c_{p}=3$. In particular, there are infinitely many elliptic curves with global Tamagawa number divisible by $3$ that are $3$-isogenous to an elliptic curve with a $3$-torsion point.

We now show that there are infinitely many elliptic curves with global Tamagawa number~$1$ that are $3$-isogenous to an elliptic curve with a $3$-torsion point. To this end, let $f(x)=(3x+1)^{3}-9$. By \cite{Erdos1953}, the set
\[
S=\left\{  x\in\mathbb{Z}\mid f(x)\text{ is squarefree}\right\}
\]
is infinite. Let $k\in S$, so that $(3k+1)^{3}-9=b$ is squarefree. Observe that $\gcd((3k+1),b)=1$ and $\gcd(3,b)=1$. Consequently, $\gcd(27(3k+1)^{3},b)=1$. Now consider the elliptic curve $\widetilde{E}_{T}=\widetilde{E}_{T}(27(3k+1)^{3},b)$. Since $27(3k+1)^{3}$ is a cube, we have by Theorem~\ref{thmontama} that the global Tamagawa number $\widetilde{c}$ of $\widetilde{E}_{T}$ is
\[
\widetilde{c}=\widetilde{c}_{3}\prod_{p|b}v_{p}(b)=\widetilde{c}_{3},
\]
since $b$ is squarefree. By construction, $27(3k+1)^{3}-27b=3^{5}$. It follows from Theorem~\ref{thmontama} that $\widetilde{c}_{3}=1$, which shows that
$\widetilde{c}=1$.
\end{proof}

Finally, we turn to the case of split Cartan subgroups. 

\begin{prop}\label{split3}
Let $E/\mathbb{Q}$ be an elliptic curve with a $3$-torsion point $P$ such that $\widetilde{E}=E/\left\langle P\right\rangle $ has a $3$-torsion point. Then, with the exception of when $E$ is $\mathbb{Q}$-isomorphic to the elliptic curve \href{https://www.lmfdb.org/EllipticCurve/Q/27/a/2}{27.a2}, it is the case that the global Tamagawa number of $\widetilde{E}$ is divisible by $3$.
\end{prop}

\begin{proof}
First, suppose that $j(E)=0$. By Proposition \ref{j0-3-torsiso}, $E$ and $\widetilde{E}$ are $\mathbb{Q}$-isomorphic to $E_{C_{3}^{0}}(1)$ and $\widetilde{E}_{C_{3}^{0}}(1)$, respectively. The LMFDB labels of these elliptic curves are \href{https://www.lmfdb.org/EllipticCurve/Q/27/a/4}{27.a4} and \href{https://www.lmfdb.org/EllipticCurve/Q/27/a/3}{27.a3}, respectively. The global Tamagawa number of these elliptic curves are $1$ and $3$, respectively. Consequently, the claim holds if $j(E)=0$.

Now suppose $j(E)\neq0$, and let $T=C_{3}$. By Proposition \ref{3torsprop}, there are relatively prime integers $a,b$ that are cubes with $a>0$ such that $E$ and $\widetilde{E}$ are $\mathbb{Q}$-isomorphic to $E_{T}(a,b)$ and $\widetilde{E}_{T}(a,b)$, respectively. Since $b$ is a cube, we have by Theorem \ref{thmontama} that for each prime $p$ dividing $b$, the local Tamagawa numbers $c_{p}$ and $\widetilde{c}_{p}$ are divisible by $3$. So it remains to consider the case when $b=\pm1$.

Since $a$ is a positive cube, there exists a positive integer $c$ such that $a=c^{3}$. Next, let $A^{\prime}=(c+6b)^{3}$ and $B^{\prime}=b(c^{2}+3cb+9b^{2})$. Let $g=\gcd(A^{\prime},B^{\prime})$, and set $A=\frac{A^{\prime}}{g}$ and $B=\frac{B^{\prime}}{g}$. Then, there is an isomorphism $\psi:\widetilde{E}_{T}(c^{3},b^{3})\rightarrow E_{T}(A,B)$ given by
\[
\psi=\left[  \frac{c^{2}g}{c^{2}+12cb+36b^{2}},c^{5}b+3c^{4}b^{2}%
,3c^{2}b,4c^{6}b^{3}\right]  .
\]
This isomorphism is verified in \cite[Proposition4\_7.ipynb]{GitHubIsog}. We claim that $g=3^{k}$ for some nonnegative integer $k$. Towards a contradiction, suppose $g\neq1$. Then there exists a prime $p$ such that $p|g$. Then $c+6b\equiv0\! \pmod{p}$. Since $b=\pm1$, we conclude that
\begin{align*}
c^{2}+3cb+9b^{2} &  =(c+6b)^{2}-9b(c+6b-3b)\\
&  \equiv27\! \pmod{p}.
\end{align*}
Thus, $p=3$. It suffices to show that the global Tamagawa number of $E_{T}(A,B)$ is divisible by $3$.

Now, suppose $q$ is a prime dividing $B$. By \cite[Theorem 2.2]{br2}, $E_{T}(A,B)$ has split multiplicative reduction at $q$ and $\widetilde{c}_{q}=3v_{p}(B)$. It follows that the global Tamagawa number $\widetilde{c}$ is divisible by $3$. It remains to consider the case when $B=\pm1$. Then $c^{2}+3cb+9b^{2}=\pm3^{k}$. If $v_{3}(c)=0$, it is easily checked that there are no integer solutions to the equation $c^{2}+3cb+9b^{2}=\pm1$. So suppose $v_{3}(c)>0$. Then $c=3l$ for some positive integer $l$. Then $l^{2}+lb+b^{2}=\pm3^{k-2}$. In fact, $k\in\left\{  2,3\right\}  $ since $l^{2}+lb+b^{2}\pmod{9}\ $is either $1,3,4,$ or $7$ modulo $9$. It is now checked that the solutions $\left(  l,b\right)  $ to $l^{2}+lb+b^{2}=\pm3^{k-2}$ for $k\in\left\{  2,3\right\}  $ are
\[
\left\{  \left(  -1,1\right)  ,\left(  0,1\right)  ,\left(  0,-1\right)
,\left(  1,-1\right)  ,\left(  -2,1\right)  ,\left(  -1,-1\right)  ,\left(
2,-1\right)  \right\}  .
\]
Since $c$ is positive, it suffices to consider $\left(  l,b\right) \in\left\{  \left(  1,-1\right)  ,\left(  2,-1\right)  \right\}  $. First, suppose $\left(  l,b\right)  =\left(  1,-1\right)  $. In this case, we obtain the elliptic curve with LMFDB label $54.a3$, which has global Tamagawa number $3$. Lastly, the solution $\left(  l,b\right)  =\left(  2,-1\right)  $ yields the elliptic curves $E$ and $\widetilde{E}$ given by LMFDB labels \href{https://www.lmfdb.org/EllipticCurve/Q/27/a/2}{27.a2} and \href{https://www.lmfdb.org/EllipticCurve/Q/27/a/4}{27.a4}, both of which have global Tamagawa number equal to $1$. This gives our one exception, which concludes the proof.
\end{proof}

The proof of Proposition~\ref{split3} leads to an alternate proof of the following curious fact that can be obtained via a quadratic reciprocity argument.

\begin{cor}
If $c$ is an integer such that $p\neq3$ is a prime dividing $c^{2}\pm3c+9$, then $p\equiv1\! \pmod{6}$.
\end{cor}

\begin{proof}
Let $b=\pm1$, so that $c^{2}\pm3c+9=c^{2}+3b+9$. Let $A$ and $B$ be as defined in the proof of Proposition \ref{split3}. By assumption, $p$ divides $B$. By \cite[Theorem 2.2]{br2}, $E_{T}$ has split multiplicative reduction at $p$. By the proof of Proposition \ref{split3}, $E_{T}(A,B)$ is $\mathbb{Q}$-isomorphic to $\widetilde{E}_{T}(c^{3},b^{3})$. Observe that $p$ divides $c^{3}-27b$. Since $\widetilde{E}_{T}(c^{3},b^{3})$ has split multiplicative reduction at $p$, it follows from Theorem~\ref{thmontama} that $p\equiv 1\! \pmod{6}$.
\end{proof}

\begin{rmk}
Suppose that $E_{1},E_{2},$ and $E_{3}$ are rational elliptic curves such that $E_{i}$ has a $3$-torsion point $P_{i}$. Suppose further that $E_{2}\cong E_{1}/\left\langle P_{1}\right\rangle $ and $E_{3}\cong E_{2}/\left\langle P_{2}\right\rangle $. Necessarily, the isogeny class degree of $E_{i}$ is divisible by $27$, and it is shown in \cite[Section 7]{MR4203041} that the only isogeny class exhibiting this over $\mathbb{Q}$ is \href{https://www.lmfdb.org/EllipticCurve/Q/27/a}{27.a}. We recover this fact as a consequence of the proof of Proposition \ref{split3}. Indeed, by the proof, we have that there exists an integer $c$ such that $E_{1}$ and $E_{2}$ are $\mathbb{Q}$-isomorphic to $E_{T}(c^{3},\pm1)$ and $\widetilde{E}_{T}(c^{3},\pm1)$, respectively. With $A$ and $B$ as defined in the proof, we have that $E_{2}$ is $\mathbb{Q}$-isomorphic to $E_{T}(A,B)$. By assumption, $E_{3}$ is $\mathbb{Q}$-isomorphic to $\widetilde{E}_{T}(A,B)$. By Proposition \ref{3torsprop}, $A$ and $B$ must be cubes since $E_{3}$ has a $3$-torsion point. It follows that $c^{2}\pm3c+9=k^{3}$ for some $k\in\mathbb{Z}$. In particular, $\left(  k,c\right)  $ is an integral solution for the elliptic curve $y^{2}\pm3y=x^{3}-9$. In both cases, this elliptic curve is $\mathbb{Q}$-isomorphic to \href{https://www.lmfdb.org/EllipticCurve/Q/27/a/3}{27.a3}. The integral solutions for $y^{2}\pm3y=x^{3}-9$ are $\left\{  \left(  3,\pm3\right)  ,\left(  3,\mp6\right)  \right\}  $. The solutions $\left(  3,3\right)  $ and $\left(  3,-3\right)  $ yield the singular curves $E_{T}(27,1)$ and $E_{T}(-27,-1)$, respectively. Whereas the solutions $\left(  3,6\right)  $ and $\left(  3,-6\right)  $ result in the elliptic curves $E_{T}(-216,1)$ and $E_{T}(216,-1)$, respectively. These elliptic curves are $\mathbb{Q}$-isomorphic and correspond to the elliptic curve with LMFDB label \href{https://www.lmfdb.org/EllipticCurve/Q/27/a/2}{27.a2}. From this, we conclude that our assumptions uniquely determine $E_{1}$ to be the elliptic curve \href{https://www.lmfdb.org/EllipticCurve/Q/27/a/2}{27.a2}; thus, there is exactly one such isogeny class exhibiting this phenomenon over $\mathbb{Q}$.
\end{rmk}

\begin{thm}\label{LoreExte}
Let $E$ be an elliptic curve with a $3$-torsion point $P$. Then, the product of the global Tamagawa numbers of the elliptic curves in the isogeny class of $E$ is divisible by~$3$.
\end{thm}

\begin{proof}
Set $\widetilde{E}=E/\left\langle P\right\rangle $. If $j(E)=0$, Proposition \ref{isomodels} implies that $E$ and $\widetilde{E}$ are $\mathbb{Q}$-isomorphic to either $\left(  i\right)  $ $E_{C_{3}}(24,1)$ and $\widetilde{E}_{C_{3}}(24,1)$, respectively, or $\left(  ii\right)  $ $E_{C_{3}^{0}}(a)$ and $\widetilde{E}_{C_{3}^{0}}(a)$, respectively, for some positive cubefree integer $a$. In the former case, we have that the theorem holds since the global Tamagawa numbers of $E_{C_{3}}(24,1)$ and $\widetilde{E}_{C_{3}}(24,1)$ are $3$ and $1$, respectively. So suppose that the latter case holds. By Theorem \ref{thmontama}, $3$ divides the global Tamagawa number of $E_{C_{3}^{0}}(a)$ whenever a prime $p\neq3$ divides $a$. So suppose that $a=3^{k}$ for some nonnegative integer $k$. If $k=0$, then the global Tamagawa number of $\widetilde{E}_{C_{3}^{0}}(a)$ is divisible by $3$. If $k$ is positive, then the global Tamagawa number of $E_{C_{3}^{0}}(a)$ is divisible by $3$. Thus, the theorem holds if $j(E)=0$.

Next, suppose that $j(E)\neq0$ and let $T=C_{3}$. By Proposition \ref{isomodels}, we have that $E$ and $\widetilde{E}$ are $\mathbb{Q}$-isomorphic to $E_{T}=E_{T}(a,b)$ and $\widetilde{E}_{T}=\widetilde{E}(a,b)$ for some relatively prime integers $a$ and $b$ with $a>0$. By Theorem \ref{thmontama}, if $\left\vert b\right\vert >1$, then the global Tamagawa number of $E_{T}$ is divisible by $3$. So we may assume that $b=\pm1$. Similarly, if $a$ is not a cube, then loc. cit. implies that the global Tamagawa number of $E_{T}$ is divisible by $3$.\ Consequently, we may assume that there is a positive integer $c$ such that $a=c^{3}$. Since $a$ and $b$ are cubes, we have that both $E_{T}$ and $\widetilde{E}_{T}$ have a $3$-torsion point by Proposition \ref{3torsprop}. By Proposition \ref{split3}, $\widetilde{E}_{T}$ has global Tamagawa number divisible by $3$, with the one exception of the elliptic curve \href{https://www.lmfdb.org/EllipticCurve/Q/27/a/4}{27.a4}. This elliptic curve is isogenous to the elliptic curve \href{https://www.lmfdb.org/EllipticCurve/Q/27/a/3}{27.a3}, which has global Tamagawa number equal to $3$.
\end{proof}

\section{Tamagawa Divisibility by  \texorpdfstring{$5$}{5} and  \texorpdfstring{$7$}{7}} \label{pthpowersection}

Let $\ell\in\left\{  5,7\right\}  $, $T=C_{\ell}$, and recall our usual conventions on $E_{T}$ and $\widetilde{E}_{T}$. In this section, we prove that there are infinitely many specializations of $\widetilde{E}_{T}$ such that its global Tamagawa number $\widetilde{c}_{T}$ is divisible by $\ell$. The strategy for both $5$ and $7$ is identical and constitutes the first result of this section. The case where the global Tamagawa number $\widetilde{c}_{T}$ is coprime to $\ell$ for infinitely many specializations is more difficult.

\begin{prop}
\label{Proppis57}Let $E/\mathbb{Q}$ be an elliptic curve with a $\ell$-torsion point $P$, where $\ell\in\left\{5,7\right\}  $. Then there are infinitely many elliptic curves $\widetilde{E}=E/\left\langle P\right\rangle $ such that its global Tamagawa number $\widetilde{c}$ is divisible by $\ell$.
\end{prop}

\begin{proof}
Let $a$ and $b$ be relatively prime positive integers such that there is a prime $q$ with $q^{\ell}|ab$. Now consider the elliptic curve $\widetilde{E}_{T}=\widetilde{E}_{T}(a,b)$ where $T=C_{\ell}$. By Theorem \ref{thmontama}, the local Tamagawa number at $q$ of $\widetilde{E}_{T}$ is divisible by $\ell$. The proposition now follows.
\end{proof}

Since our goal is to understand the $\ell$-divisibility of the elements of $\mathcal{L}_{\ell}\backslash\mathcal{G}_{\ell}$, we must consider the case in Proposition \ref{Proppis57} corresponding to $\widetilde{E}=E/\left\langle P\right\rangle $ having a $5$-torsion point. We note that when $\ell=7$, there is no issue since there are no elliptic curves with a $7$-torsion point whose mod~$7$ representation is a split Cartan subgroup of $\operatorname*{GL}\nolimits_{2}(\mathbb{F}_{7})$. 

By Lorenzini \cite{lorenzini}, there is exactly one elliptic curve with a $5$-torsion point that has global Tamagawa number not divisible by $5$. Consequently, if an isogeny class contains two elliptic curves with a $5$-torsion point, then it follows that with at most one exception, the product of the global Tamagawa numbers of the elliptic curves in said isogeny class is divisible by $25$. Below, we provide a proof of this consequence by means of the parameterizations $E_T$ and $\widetilde{E}_T$, as the proof informs some of the computations we discuss immediately after this result.

\begin{cor}\label{5cartan}
Let $E/\mathbb{Q}$ be an elliptic curve with a $5$-torsion point $P$. Suppose further that $\widetilde{E}=E/\left\langle P\right\rangle $ has a $5$-torsion point. Then, with the exception of the isogeny class with LMFDB label \href{https://www.lmfdb.org/EllipticCurve/Q/11/a/}{11.a}, the product of the global Tamagawa numbers of the elliptic curves in the isogeny class of $E$ is divisible by $25$.
\end{cor}

\begin{proof}
By \cite[Lemma~4.3.3]{ckv}, we may assume that there are relatively prime integers $s$ and $t$ such that $E$ and $\widetilde{E}$ are $\mathbb{Q}$-isomorphic to $E_{T}=E_{T}(s^{5},t^{5})$ and $\widetilde{E}_{T}=\widetilde{E}_{T}(s^{5},t^{5})$, respectively. If $st\neq\pm1$, then the global Tamagawa numbers of both $E_{T}$ and $\widetilde{E}_{T}$ is divisible by $5$. So it suffices to consider the cases corresponding to $st=\pm1$. It is then checked that the cases corresponding to $st=\pm1$ result in the isogeny class \href{https://www.lmfdb.org/EllipticCurve/Q/11/a/}{11.a}, which has the property that the product of the global Tamagawa numbers of the elliptic curves in the isogeny class is $5$.
\end{proof}

To finish this section, we consider the special case where both ${E}_{T}(a,b)$ and $\widetilde{E}_{T}(a,b)$ have a point of order $5$, so that $\im \overline{\rho}_{E,5}$ is a split Cartan subgroup of $\GL_2(\F_5)$.   Since our aim in this paper is primarily to study the effect of $\ell$-isogeny on the global Tamagawa number when the isogenous curve does \emph{not} have a $\ell$-torsion point, we do not pursue the split Cartan case in detail.  Rather, we simply end by writing down the discriminant formulas and observing that one could engineer families of specializations with Tamagawa numbers of high 5-adic valuation.  Since $\gcd(a,b)=1$, we may write $a= s^5$ and $b=t^5$.  Define
\begin{align*}
d_1(s,t) &= s^2 + st - t^2 \\
d_2(s,t) &= s^4 - 3s^3t + 4s^2t^2 - 2st^3 + t^4 \\
d_3(s,t) &= s^4 + 2s^3t + 4s^2t^2 + 3st^3 + t^4.
\end{align*}
Then we compute
\begin{align*}
\Delta_T(s^5,t^5) &= -s^{25}t^{25}d_1(s,t)d_2(s,t)d_3(s,t) \\
\widetilde{\Delta}_T(s^5,t^5) &= -s^{5}t^{5}d_1(s,t)^5d_2(s,t)^5d_3(s,t)^5.
\end{align*}

\begin{rmk}
Experimentally, $c_{\widetilde{E}_{T}(a,b)}$ appears to be \emph{more} divisible by $5$ than $c_{E_{T}(a,b)}$.  This is likely due to the fact that the $\widetilde{\Delta}_{T}(s^5,t^5)$ factors as
\[
-s^5t^5(s^2 + st - t^2)^5(s^4 - 3s^3t + 4s^2t^2 - 2st^3 + t^4)^5(s^4 + 2s^3t + 4s^2t^2 + 3st^3 + t^4)^5.
\]
For any split multiplicative prime dividing $(s^2 + st - t^2)$, $(s^4 - 3s^3t + 4s^2t^2 - 2st^3 + t^4)$, or $(s^4 + 2s^3t + 4s^2t^2 + 3st^3 + t^4)$, the Tamagawa number $c_{\widetilde{E}_{T}(a,b)}$ picks up an additional factor of 5. 
 Then, we performed a simple computer calculation. For each $(a,b) \in [2,1000] \times [2,1000]$ with $\gcd(a,b)=1$, we found $\mathbf{18}$ cases in which $c_{\widetilde{\Delta}_{T}(a^5,b^5)} \not \equiv 0 \pmod{5^5}$ and $\mathbf{36180}$ cases in which $c_{{\Delta}_{T}(a^5,b^5)} \not \equiv 0 \pmod{5^5}$. 
\end{rmk}

\section{Tamagawa Indivisibility by  \texorpdfstring{$5$}{5} and  \texorpdfstring{$7$}{7}} \label{conditional_proofs}

While it was straightforward to produce infinitely many specializations of $\widetilde{E}_{T}(a,b)$ without a $\ell$-torsion point and with $\widetilde{c}$ divisible by $\ell$, it is more difficult to show that there are infinitely many specializations for which $\widetilde{c}$ is coprime to $\ell$. To investigate this further, set
\[
\left(  n_{\ell}(a,b),f_{\ell}(a,b)\right)  =\left\{
\begin{array}
[c]{cl}%
\left(  ab,a^{2}+11ab-b^{2}\right)   & \text{if }\ell=5,\\
\left(  ab(a-b),a^{3}+5a^{2}b-8ab^{2}+b^{3}\right)   & \text{if }\ell=7.
\end{array}
\right.
\]
With this notation, observe that the discriminant $\Delta_{T}$ and $\widetilde{\Delta}_{T}$ of $E_{T}$ and $\widetilde{E}_{T}$, respectively, are given by
\[
\Delta_{T}=n_{\ell}(a,b)^{\ell}f_{\ell}(a,b)\qquad\text{and}\qquad\widetilde{\Delta
}_{T}=n_{\ell}(a,b)f_{\ell}(a,b)^{\ell}.
\]
By Theorem \ref{thmontama}, to engineer infinitely many specializations of $\widetilde{E}_{T}=\widetilde{E}_{T}(a,b)$ with global Tamagawa number coprime to $\ell$, it must be the case that $\widetilde{E}_{T}$ have non-split multiplicative reduction at each prime $q\neq \ell$ dividing $f_{\ell}(a,b)$. To this end, we introduce the following notation for ease of reference to Theorem~\ref{thmontama}:
\[
g_{\ell}(a,b)=\left\{
\begin{array}
[c]{cl}%
a^{2}+b^{2} & \text{if }\ell=5,\\
a^{2}-ab+b^{2} & \text{if }\ell=7.
\end{array}
\right.
\]
With all of this setup, we now have the following characterization of those $\widetilde{c}$ that are coprime to $\ell$.

\begin{thm}
\label{coprimethm}Let $T=C_{\ell}$ for $\ell\in\left\{  5,7\right\}  $, and let $E/\mathbb{Q}$ be an elliptic curve with a $\ell$-torsion point $P$. Then the global Tamagawa number of $\widetilde{E}=E/\left\langle P\right\rangle $ is coprime to $\ell$ if and only if there exists positive relatively prime integers $a$ and $b$ such that $n_{\ell}(a,b)$ is $\ell$-th power-free such that $E$ and $\widetilde{E}$ are $\mathbb{Q}$-isomorphic to $E_{T}=E_{T}(a,b)$ and $\widetilde{E}_{T}(a,b)$, respectively, and for each prime $q\neq \ell$ dividing $f_{\ell}(a,b)$, it is the case that
\begin{equation}
\left(  \frac{-\ell g_{\ell}(a,b)}{q}\right)  =-1.\label{quad}%
\end{equation}

\end{thm}

\begin{proof}
By Proposition \ref{isomodels}, we may assume that there are relatively prime integers with $a$ positive such that $E$ and $\widetilde{E}$ are $\mathbb{Q}$-isomorphic to $E_{T}=E_{T}(a,b)$ and $\widetilde{E}_{T}(a,b)$, respectively. If $b$ is negative, then
\[
\phi_{\ell}=\left\{
\begin{array}
[c]{cl}%
\left[  1,ab,-a,0\right]   & \text{if }\ell=5,\\
\left[  1,ab^{3}-a^{2}b^{2},b^{2}-a^{2},a^{3}b^{3}-2a^{2}b^{4}+ab^{5}\right]
& \text{if }\ell=7,
\end{array}
\right.
\]
is an isomorphism from $E_{T}$ to $E_{T}(a,-b)$ (resp. $E_{T}(-b,a-b))$ if $\ell=5$ (resp. $\ell=7$). This shows that we may assume $a$ and $b$ to be positive relatively prime positive integers. Since the global Tamagawa number $\widetilde{c}_{T}$ of $\widetilde{E}_{T}$ is coprime to $\ell$, it must be the case by Theorem \ref{thmontama} that $n_{\ell}(a,b)$ is $\ell$-th power-free. In addition, we must also have that for each prime $q\neq \ell$ dividing $f_{\ell}(a,b)$, it is the case that (\ref{quad}) holds. This concludes the forward direction, and the converse is automatic from Theorem~\ref{thmontama}.
\end{proof}

It remains to engineer infinitely many pairs $(a,b)$ satisfying the hypotheses of Theorem~\ref{coprimethm}.  As the number of prime divisors of $f_{\ell}(a,b)$ grows, this becomes more and more difficult to satisfy.    Thus, the simplest way to satisfy these hypotheses is to show that there are infinitely many coprime, $\ell$th-power-free $(a,b)$ such that $f_{\ell}(a,b) = q$ is a prime number and $\left( \frac{-\ell g_\ell(a,b)}{q} \right) =-1$.  We discuss two special cases.

\bigskip

\noindent \textbf{\fbox{Case 1: $\ell=5$}}  We start with a lemma.

\begin{lem}
\label{LegendreQuad}There are infinitely many coprime integer solutions $\left(  a,b\right)  $ satisfying
\begin{align}
a^{2}+11ab-b^{2}  & =19.\label{eq19}%
\end{align}
Moreover, each solution $\left(  a,b\right)  \in \Z^{2}$ to \eqref{eq19} satisfies $\left(  \frac
{-5(a^{2}+b^{2})}{19}\right)  =-1$.
\end{lem}

\begin{proof}
Note that $(a,b) = (1,2)$ is a solution to (\ref{eq19}). Thus, there are there are infinitely many integer solutions by the theory of Pell conics. Now let $(a,b)$ be an arbitrary solution to (\ref{eq19}), suppose that $d=\gcd(a,b)$, and write $a=d\bar{a}$ and $b=d\bar{b}$. Then
\[
d^{2}\left(  \bar{a}^{2}+11\bar{a}\bar{b}-\bar{b}^{2}\right)  =19.
\]
This implies that $d^{2}$ divides $19$, and thus $\gcd\left(  a,b\right)  =1$. It remains to show that the Legendre symbol is as claimed.  Now we compute 
\[
a^{2}+11ab-b^{2}\equiv\left(  a+2b\right)  \left(  a+9b\right)  \ \left(
\operatorname{mod}19\right)  =0\ \left(  \operatorname{mod}19\right)  .
\]
Consequently, $a=-2b$\ $\left(  \operatorname{mod}19\right)  $ or
$a=-9b\ \left(  \operatorname{mod}19\right)  $. The proof now follows since
substituting into the Legendre symbol yields
\[
\left(  \frac{-5\left(  a^{2}+b^{2}\right)  }{19}\right)  =\left\{
\begin{array}
[c]{cl}%
\left(  \frac{-5\left(  4b^{2}+b^{2}\right)  }{19}\right)  =\left(  \frac
{-1}{19}\right)  =-1 & \text{if }a=-2b\ \left(  \operatorname{mod}19\right)
,\\
\left(  \frac{-5\left(  81b^{2}+b^{2}\right)  }{19}\right)  =\left(  \frac
{8}{19}\right)  =-1 & \text{if }a=-9b\ \left(  \operatorname{mod}19\right)  .
\end{array}
\right. \qedhere
\]
\end{proof}

We would like to use exactly these pairs $(a,b)$ as specializations to produce infinitely many $\widetilde{c}$ coprime to $5$, but we are currently unable to prove that infinitely many of them are also 5th-power-free.  If we define the set 
\[
X_5 \ddef \left\{ \left(  a,b\right)  \in \Z^{2}\mid\gcd(a,b)=1,a^{2}+11ab-b^{2}=19,\text{ and }a,b\text{ 5th-power-free}%
\right\},
\]
then we have the following corollary of Theorem \ref{coprimethm}, contingent upon the infinitude of $X_5$.

\begin{cor} \label{5coprimecor}
If $X_5$ is infinite, then there are infinitely many elliptic curves $E$ with a $5$-torsion point $P$ such that $\widetilde{E}=E/\left\langle P\right\rangle $ has global Tamagawa number coprime to $5$.
\end{cor}

\begin{proof}
The infinitude of $X_5$ furnishes us with infinitely many specializations satisfying the hypotheses of Theorem \ref{coprimethm}.
\end{proof}

\bigskip

\noindent \textbf{\fbox{Case 2: $\ell=7$}} When $\ell=7$, the situation is more complicated, owing to the fact that $f_7(a,b)$ is cubic.  In the previous case, we were able to easily show using the theory of Pell conics that for a fixed prime $q$ (we used $q=19$), there are infinitely many solutions to $f_5(a,b) = q$, provided there exists one solution.  Then apply Corollary \ref{5coprimecor}.  

We still wish to consider solutions to $f_7(a,b) = q$ prime since that is the simplest way to apply Theorem \ref{coprimethm}.  Toward this end, we make several observations.

\begin{lem}\label{lem7torsspeind}
With notation as above, suppose that $f_{7}(a,b)=q$ is prime. Then $q\equiv\pm1\ (\operatorname{mod}7)$, and $q\equiv\pm1\ (\operatorname{mod}7)$ if and
only if
\[
\left(  \frac{-7g_{7}(a,b)}{q}\right)  =\pm1.
\]
\end{lem}

\begin{proof}
The proof follows from arithmetic with congruences and several applications of quadratic reciprocity.  We omit the details.  
\end{proof}

Let $t \in \Z$.  Then the solutions to $f_7(a,b) = t$ can be viewed as affine points on the projective cubic
\begin{align} \label{j0}
f_7(a,b)  = tz^3.
\end{align}
If there exist solutions to (\ref{j0}), then it defines an elliptic curve over $\Q$ (note that the $j$-invariant must be 0 since $z \to e^{2\pi i/3} z$ is an automorphism of order 3).  A standard chord and tangent argument shows that the elliptic curve defined by (\ref{j0}) is birational to $E_{C_3}^0(49t)$, using our notation from above.

Therefore, fixing $t=q \equiv -1 \pmod{7}$, we can only expect finitely many solutions to $f_7(a,b) = q$.   Thus, the identical reasoning as in the case $p=5$ will not produce a (conjecturally) infinite set $X_5$ for us to apply Theorem \ref{coprimethm}. However, Lemma \ref{lem7torsspeind} gives us instructions for an alternate set to consider. Namely,
\begin{small}
\[
X_{7}=\left\{  \left(  a,b\right)  \in\mathbb{Z}^{2}\mid\gcd(a,b)=1,ab(a-b)\text{ is }7\text{-th-power-free, }f_{7}(a,b)\equiv-1\! \pmod{7}\text{ is prime}\right\}
\]

\end{small}

\begin{cor}\label{7coprimecor}
If $X_{7}$ is infinite, then there are infinitely many elliptic curves $E$ with a $7$-torsion point $P$ such that $\widetilde{E}=E/\left\langle P\right\rangle $ has global Tamagawa number coprime to $7$.
\end{cor}

\begin{proof}
By Lemma \ref{lem7torsspeind}, the infinitude of $X_{7}$ yields infinitely many specializations satisfying the hypothesis of Theorem \ref{coprimethm}.
\end{proof}

We perform another simple experiment: running over all coprime $a$ and $b$ for $1\leq a,b \leq 1000$ we find \textbf{32139} distinct values of $q$ such that $f_7(a,b) = q \equiv -1 \pmod{7}$ and $ab(a-b)$ is 7-th power-free.    We expect that there are infinitely many $q$ that can be written as $f(a,b)$, though we are unable to prove it.  If so, then it would immediately imply that there are infinitely many specializations $\widetilde{E}_T(a,b)$ that have $\widetilde{c}$ coprime to 7. In the next section we go into some more detail on the computational aspects of this project. 

\section{Computational Results} \label{computations}

\subsection{Background}
One way to strengthen the results of this paper would be to determine whether or not there exists a well-defined proportion of elliptic curves in $\mathcal{L}_{\ell} \smallsetminus \mathcal{G}_{\ell}$ with global Tamagawa number divisible by $\ell$.  This was the point of view taken in \cite{ckv}, where the authors studied the relative proportion of $\mathcal{G}_{\ell}$ in $\mathcal{L}_{\ell}$, as we now recall.

Ordering all elliptic curves over $\Q$ by height, set $\mathcal{L}_{\ell,X}$ to be the number of elliptic curves of height $\leq X$ that locally have a subgroup of order $\ell$, and $\mathcal{G}_{\ell,X}$ the number of height $\leq X$ that globally have a subgroup of order $\ell$.  By \cite[Thm.~1.3.1]{ckv} the limit $\mathcal{P}_\ell \ddef \lim_{X \to \infty} |\mathcal{G}_{\ell,X}|/|\mathcal{L}_{\ell,X}|$ exists, and can be interpreted as the conditional probability that an elliptic curve has a global subgroup of order $\ell$ given that it locally has a subgroup of order $\ell$.  For instance,
\[
\mathcal{P}_{3} = 1/2, \ \mathcal{P}_{5} = 25/34, \text{ and }\mathcal{P}_{7} = 4/(4+\sqrt{7}) \approx 0.602.
\]

However, we are unable at this point to predict whether a similar proportion exists for the Tamagawa data we study in this paper.  More precisely, we set 
\[
\mathcal{T}_{\ell,X} \ddef \frac{| \lbrace E \in \mathcal{L}_{\ell,X} \smallsetminus \mathcal{G}_{\ell,X} ~:~ c_E \equiv 0 \pmod{\ell}\rbrace|}{|\mathcal{L}_{\ell,X} \smallsetminus \mathcal{G}_{\ell,X}|}.
\]
In \cite{hs}, the authors determined asymptotics for the number of elliptic curves of height $\leq X$ with a given torsion subgroup.  (For example, they show that $|\mathcal{G}_{5,X}| = cX^{1/6} + O(X^{1/12})$, where $c$ is a computable constant.)  In \cite{ckv} the authors showed that $|\mathcal{L}_{\ell,X}|$ has the same growth asymptotics and error terms (though with different leading constant), allowing them to compute the limits as $X \to \infty$ as ratios of growth constants.  

Two items that make proving analogous results for $\mathcal{T}_{\ell,X}$ more difficult are
\begin{enumerate}
\item it is not clear what the correct order of growth is for the denominator $|\mathcal{L}_{\ell,X} \smallsetminus \mathcal{G}_{\ell,X}|$, and  
\item as the parameters $a$ and $b$ range over positive, coprime integers, we would need asymptotics on quadratic and cubic forms in $a$ and $b$ with certain prime factorizations.  
\end{enumerate}

With all of that said, we give some data below on $\mathcal{T}_{\ell,X}$ for several values of $X$.  We observe that the proportions do not appear to be converging to a limit that we can discern from these data.

Note that for $\ell\in\left\{  5,7\right\}  $, Proposition \ref{Proppis57} gives infinitely many specializations of $\widetilde{E}_{C_{\ell}}$ whose global Tamagawa number is divisible by $\ell$. Theorem \ref{coprimethm} gives necessary and sufficient conditions on the parameters of $\widetilde{E}_{C_{\ell}}$ for determining those specializations whose global Tamagawa number is coprime to $\ell$. While we are unable to prove that there are infinitely many such specializations, Corollaries \ref{5coprimecor} and \ref{7coprimecor} show that there are infinitely many such specializations provided that the set $X_{\ell}$ is infinite. Motivated by this, we investigate the $\mathcal{T}_{\ell,X}$ by making use of Theorem \ref{coprimethm}.

For an elliptic curve $E/\mathbb{Q}$ given by a global minimal model, the naive height of $E$ is
\begin{align}
\operatorname*{ht}(E)= \frac{1}{12}\log\max\left\{  \left\vert c_{4}^{3}\right\vert
,c_{6}^{2}\right\} \label{log_ht} .
\end{align}
By \cite[Theorem 6.1]{Barrios}, $E_{C_{\ell}}$ is given a global minimal model provided that the parameters $a$ and $b$ are relatively prime. This is not necessarily true for $\widetilde{E}_{C_{\ell}}$. Consequently, our investigation
of $\mathcal{T}_{\ell,X}$ will be done by considering the naive height of $E_{C_{\ell}}$, instead of $\widetilde{E}_{C_{\ell}}$. To this end, we make the following observation:

\begin{lem}
\label{LemKatz}Let $\ell\in\left\{  3,5,7\right\}  $. For each $\widetilde{E}\in\mathcal{L}_{\ell}\backslash\mathcal{G}_{\ell}$, there exists a unique elliptic curve $E\in\mathcal{G}_{\ell}$ such that $\widetilde{E}$ is $\mathbb{Q}$-isomorphic to $E/\left\langle P\right\rangle $, where $P\in E(\mathbb{Q})$ is a point of order $\ell$.
\end{lem}

\begin{proof}
This follows from \cite[Lemma~1]{Katz}.
\end{proof}

For $E\in\mathcal{G}_{\ell}$, let $P_{\ell}\in E(\mathbb{Q})$ denote a point of order $\ell$. Now set $\widetilde{E}=E/\left\langle P_{\ell}\right\rangle $. By Lemma \ref{LemKatz}, we have that there is a one-to-one correspondence between $\mathcal{L}_{\ell}\backslash\mathcal{G}_{\ell}$ and $\widetilde{G}_{\ell}=\left\{  E\in\mathcal{G}_{\ell}\mid\widetilde{E}\in\mathcal{L}_{\ell}\backslash\mathcal{G}_{\ell}\right\}  $. Now let $\widetilde{\mathcal{G}}_{\ell,X}$ denote the number of elliptic curves in $\widetilde{G}_{\ell}$ with height $\leq X$. With this notation, we aim to investigate $\mathcal{T}_{\ell,X}$ by considering the set
\[
\widetilde{\mathcal{T}}_{\ell,X}=\frac{\left\vert \left\{  E\in
\mathcal{\widetilde{G}}_{\ell,X}\mid c_{\widetilde{E}}\equiv
0\ (\operatorname{mod}\ell)\right\}  \right\vert }{\left\vert
\mathcal{\widetilde{G}}_{\ell,X}\right\vert }.
\]
When $\ell\in\left\{  5,7\right\}  $, Theorem \ref{coprimethm} allows us to investigate $\widetilde{\mathcal{T}}_{\ell,X}$ in terms of the parameters of $E_{C_{\ell}}$. Table~\ref{ta:Stats} summarizes our findings, with $\mathcal{N}_{\ell,X}=\left\{  E\in\mathcal{\widetilde{G}}_{\ell,X}\mid c_{\widetilde{E}}\equiv0\ \operatorname{mod}\ell)\right\}  $.

{\begingroup \footnotesize\renewcommand*{\arraystretch}{1.25}
\begin{longtable}{crrc||crrc}
\caption{The quantities $\mathcal{N}_{\ell,X}$, $\mathcal{\widetilde{G}}_{\ell,X}$, and $100\cdot\widetilde{\mathcal{T}}_{\ell,X}$}\\
\hline
$X$ & $\mathcal{N}_{5,X}$ & $\mathcal{\widetilde{G}}_{5,X}$ & $100\cdot
\widetilde{\mathcal{T}}_{5,X}$ & $X$ & $\mathcal{N}_{7,X}$ &
$\mathcal{\widetilde{G}}_{7,X}$ & $100\cdot\widetilde{\mathcal{T}}_{7,X}$\\
\hline
\endfirsthead
\caption[]{\emph{continued}}\\
\hline
$X$ & $\mathcal{N}_{5,X}$ & $\mathcal{\widetilde{G}}_{5,X}$ & $100\cdot
\widetilde{\mathcal{T}}_{5,X}$ & $X$ & $\mathcal{N}_{7,X}$ &
$\mathcal{\widetilde{G}}_{7,X}$ & $100\cdot\widetilde{\mathcal{T}}_{7,X}$\\
\hline
\endhead
\hline
\multicolumn{2}{r}{\emph{continued on next page}}
\endfoot
\hline
\endlastfoot

$4$ & $581$ & $958$ & $60.65\%$ & $13$ & $113\ 594$ & $162\ 440$ & $69.93\%$\\\hline
$5$ & $4\ 611$ & $7\ 105$ & $64.90\%$ & $14$ & $313\ 490$ & $441\ 616$ &
$70.99\%$\\\hline
$6$ & $35\ 603$ & $52\ 412$ & $67.93\%$ & $15$ & $864\ 709$ & $1\ 200\ 358$ &
$72.04\%$\\\hline
$7$ & $272\ 541$ & $387\ 350$ & $70.36\%$ & $16$ & $2\ 381\ 512$ &
$3\ 262\ 925$ & $72.99\%$\\\hline
$8$ & $2\ 067\ 893$ & $2\ 862\ 231$ & $72.25\%$ & $17$ & $6\ 547\ 843$ &
$8\ 869\ 583$ & $73.82\%$\\\hline
$9$ & $15\ 613\ 608$ & $21\ 144\ 556$ & $73.84\%$ & $18$ & $17\ 978\ 901$ &
$24\ 110\ 094$ & $74.57\%$\\\hline
$9.5$ & $42\ 839\ 637$ & $57\ 484\ 021$ & $74.52\%$ & $19$ & $49\ 324\ 629$ &
$65\ 537\ 998$ & $75.26\%$\\\hline
$10$ & $117\ 457\ 574$ & $156\ 265\ 062$ & $75.17\%$ & $20$ & $133\ 946\ 795$
& $176\ 476\ 581$ & $75.9\%$%

\label{ta:Stats}	
\end{longtable}\endgroup}

Consider the columns $\widetilde{\mathcal{G}}_{5,X}$ and $\widetilde{\mathcal{G}}_{7,X}$.  In \cite{ckv}, the authors  proved the following asymptotic formulas
\begin{align*}
\widetilde{\mathcal{G}}_{5,X} &= \widetilde{\gamma}_5X^{1/6} + O(X^{1/12}),\text{ and} \\
\widetilde{\mathcal{G}}_{7,X} &= \widetilde{\gamma}_7X^{1/12} + O(X^{1/24})
\end{align*}
for explicit constants $\widetilde{\gamma}_5$ and $\widetilde{\gamma}_7$.  Therefore, we would expect to see in the data above that the quotients
\begin{align*}
\frac{\mathcal{G}_{5,10X}}{\mathcal{G}_{5,X}} &= 10^{1/6} \approx 1.4678, \text{ and} \\
\frac{\mathcal{G}_{7,10X}}{\mathcal{G}_{7,X}} &= 10^{1/12} \approx 1.2115
\end{align*}
are constant.  Because we used the logarithmic height (\ref{log_ht}) to generate our data, versus the height function 
\[
{\rm Ht}(y^2 = x^3 + ax+b) = \max (\abs{4a^3},27b^2)
\]
of \cite{ckv}, we cannot expect to see the ratios $10^{1/6}$ and $10^{1/12}$ in our data, but we should see that the ratios appear to be constant if we replace $X$ with $10X$.  This is indeed the case:
\begin{center}
    \begin{table}[h]
\begin{tabular}{rr||rr} 
$X$ & $\widetilde{\mathcal{G}}_{5,10X}/\widetilde{\mathcal{G}}_{5,X}$ & $X$ &    $\widetilde{\mathcal{G}}_{7,10X}/\widetilde{\mathcal{G}}_{7,X}$\\
\hline
4 & 7.4165 & 14 & 2.6197 \\
5 & 7.3768 & 15 & 2.6260 \\
6 & 7.3905 & 16 & 2.6341 \\
7 & 7.3893 & 17 & 2.6408\\
8 & 7.3874 & 18 & 2.6444\\
9 & 7.3903 & 19 & 2.6231
\end{tabular}
\caption{Ratios $\widetilde{\mathcal{G}}_{\ell,10X}/\widetilde{\mathcal{G}}_{\ell,X}$}
\end{table}
\end{center}

\vspace{-3em}

We do not claim to have an explicit asymptotic for the numbers $\mathcal{N}_{\ell,X}$, but we can perform the same experiment on the ratios to see that they, too, appear to be constant:

\begin{center}
    \begin{table}[h]
\begin{tabular}{rr||rr} 
$X$ & $\widetilde{\mathcal{N}}_{5,10X}/\widetilde{\mathcal{N}}_{5,X}$ & $X$ &    $\widetilde{\mathcal{N}}_{7,10X}/\widetilde{\mathcal{N}}_{7,X}$\\
\hline
4 & 7.9363 & 14 & 2.7583\\
5 & 7.7213 & 15 & 2.7541 \\
6 & 7.6550 & 16 & 2.7494 \\
7 & 7.5875 & 17 & 2.7458\\
8 & 7.5505 & 18 & 2.7435\\
9 & 7.5228 & 19 & 2.7156
\end{tabular}
\caption{Ratios $\widetilde{\mathcal{N}}_{\ell,10X}/\widetilde{\mathcal{N}}_{\ell,X}$}
\end{table}
\end{center}

\vspace{-3em}

We conclude with a discussion of how the data in Table \ref{coprimethm} was computed. Let $T=C_{\ell}$ for $\ell\in\left\{  5,7\right\}  $, and let $\alpha_{T}=\alpha_{T}(a,b)$ and $\beta_{T}=\beta_{T}(a,b)$ denote the invariants $c_{4}$ and $c_{6}$, respectively, of $E_{T}$. These polynomials are given explicitly in Tables 4 and 5, respectively, of \cite{Barrios}. Since $E_{T}$ is given by a global minimal model, we have that
\[
\operatorname*{ht}(E_{T})=\frac{1}{12}\max\log\left\{  \left\vert \alpha
_{T}^3\right\vert ,\beta_{T}^2\right\}  .
\]
Next, let $\mathcal{E}_{T}\ $denote the set of $\mathbb{Q}$-isomorphism classes of elliptic curves $E$ such that $T\hookrightarrow E(\mathbb{Q})$. Then, we have a surjective map $\psi:\left\{  \left(  a,b\right)  \in\mathbb{N}^{2}\mid\gcd(a,b)=1\right\}  \rightarrow\mathcal{E}_{T}$ given by $\psi(a,b)=E_{T}(a,b)$. We note that this map is not necessarily injective as distinct pairs $\left(  a,b\right)  $ and $\left(  a^{\prime},b^{\prime}\right)  $ may yield the same $\mathbb{Q}$-isomorphic elliptic curve. However, such duplicates can be removed by considering the reduced minimal models of $\psi(a,b)$ and $\psi(a^{\prime },b^{\prime})$. Here, we recall that the reduced minimal model of an elliptic curve is unique up to $\mathbb{Q}$-isomorphism \cite{MR1628193}. In fact, \cite{barrios2023red} gives necessary and sufficient conditions on the parameters of $E_{T}$ to determine the reduced minimal model of $E_{T}$. This gives rise to a one-to-one correspondence between $\mathcal{E}_{T}$ and $\left\{  \left(  a,b\right)  \in\mathbb{N}^{2}\mid\gcd(a,b)=1\right\}  /\sim_{T}$, where $\sim_{T}$ is the equivalence relation defined by $\left(  a,b\right)  \sim_{T}\left(  a^{\prime},b^{\prime }\right)  $ if and only if $E_{T}(a,b)$ and $E_{T}(a^{\prime},b^{\prime})$ are $\mathbb{Q}$-isomorphic.

Next, let
\[
\Lambda_{T}=\left\{
\begin{array}
[c]{cl}%
\left\{  \left(  a,b\right)  \in\mathbb{N}^{2}\mid\gcd(a,b)=1,ab\text{ not a }5\text{th-power}\right\}  /\sim_{T} &
\text{if }T=C_{5},\\
\left\{  \left(  a,b\right)  \in\mathbb{N}^{2}\mid\gcd(a,b)=1\right\}  /\sim_{T} & \text{if }T=C_{7}.
\end{array}
\right.
\]
By the above, Lemma \ref{LemKatz}, and the fact that there is no isogeny class over $\mathbb{Q}$ consisting of two distinct elliptic curves with a $7$-torsion point, we have that $\mathcal{\widetilde{G}}_{7}$ is bijective to $\Lambda_{C_{7}}$. For $T=C_{5}$, we have by \cite[Lemma 4.3.3]{ckv} that $E_{T}$ and $\widetilde{E}_{T}$ both have a $5$-torsion point if and only if $ab$ is a fifth power. This, coupled with our above discussion and Lemma \ref{LemKatz} gives a one-to-one correspondence between $\mathcal{\widetilde{G}}_{5}$ and $\Lambda_{C_{5}}$. Finally, the explicit description of $\operatorname*{ht}(E_{T})$ allows us to demonstrate that if $\operatorname*{ht}(E_{T})<10$ (resp. $20$) for $T=C_{5}$ (resp. $C_{7}$), then $a+b<55\ 000$. So to compute the data in Table \ref{coprimethm}, we first compute $\left\{  \left(a,b\right)  \in\mathbb{N}^{2}\mid\gcd(a,b)=1,a+b<55\ 000\right\}  $. For each pair in this set, we compute the reduced minimal model of the corresponding elliptic curves. After removing duplicates, we obtain subsets $\widetilde{\Lambda}_{T}$ that are a complete set of representatives for $\widetilde{\Lambda}_{T}$. In particular, we obtain a subset in one-to-one correspondence with $\mathcal{\widetilde{G}}_{\ell,X}$ for $X=10$ (resp. $20$) if $T=C_{5}$ (resp. $C_{7})$. Theorem \ref{coprimethm} then gives us the required assumptions to find the subset of $\widetilde{\Lambda}_{T}$ that is
in one-to-one correspondence with $\mathcal{N}_{\ell,X}$.

\bibliographystyle{plain}
\bibliography{main}

\end{document}